\documentclass[a4paper,10pt,DIV11]{scrartcl}
\usepackage[utf8]{inputenc}

\usepackage{amsfonts}
\usepackage{amsmath,amssymb,amsthm}

\usepackage{stmaryrd}
\usepackage{physics}
\usepackage{tikz}
\usepackage{siunitx}
\usepackage{subcaption}
\usepackage{hyperref}
\usepackage{verbatim}
\usepackage{cite}
\usepackage{placeins}
\usepackage{color,wrapfig}
\usepackage{calligra}
\usepackage{mathtools}
\usepackage{multirow}
\usepackage{booktabs}

\usepackage{todonotes}
\usepackage{soul}

\newcommand*{\ldblbrace}{\lbrace\mskip-5mu\lbrace}
\newcommand*{\rdblbrace}{\rbrace\mskip-5mu\rbrace}
\newcommand{\jump}[1]{\left\llbracket #1 \right\rrbracket}
\newcommand{\jp}[1]{\left[ #1 \right]}
\newcommand{\average}[1]{\left\ldblbrace #1 \right\rdblbrace}
\newcommand{\scp}[2]{\left\langle #1, #2 \right\rangle}
\newcommand{\abso}[1]{\left\vert #1 \right\vert}

\newcommand{\normLone}[1]{\left\lVert #1 \right\rVert_{L^1}}

\newcommand{\Ff}{\mathcal{F}}
\newcommand{\Ii}{\mathcal{I}}
\newcommand{\RR}{\mathbb{R}}
\newcommand{\Th}{\mathcal{T}_h}
\newcommand{\kk}{k}
\newcommand{\Kone}{k_{1}}
\newcommand{\Ktwo}{k_{2}}
\newcommand{\Kcut}{\text{cut}}
\newcommand{\jj}{j}
\newcommand{\ubar}{\bar{u}}
\newcommand{\xjplus}{x_{\jj+\frac{1}{2}}^+}
\newcommand{\xjminus}{x_{\jj+\frac{1}{2}}^-}

\newcommand{\GammaInt}{\Gamma_{\text{int}}}

\newcommand{\OmegaBg}{\widehat{\Omega}}
\newcommand{\ThBg}{\widehat{\mathcal{T}_h}}
\newcommand{\lTE}{\operatorname{len}(T_E)}  
\newcommand{\NeighE}{\mathcal{N}_o(E)}
\newcommand{\indNeighE}{\mathbb{I}_{\NeighE}}
\newcommand{\indNeighEplusE}{\mathbb{I}_{E \cup \NeighE}}

\newcommand{\etaE}{\eta_E}
\newcommand{\etaKone}{\eta_{\Kone}}
\newcommand{\Deltat}{\Delta t}
\newcommand{\diagentry}[1]{\text{\fboxsep.5ex\fbox{$#1$}}}

\newtheorem{remark}{Remark}[section]
\newtheorem{definition}{Definition}[section]
\newtheorem{lemma}{Lemma}[section]
\newtheorem{corollary}{Corollary}[section]

\definecolor{mid2-gray}{gray}{0.55}

\title{A stabilized DG cut cell method for discretizing the linear transport equation}
\author{Christian Engwer\thanks{Applied Mathematics M\"unster: Institute for Analysis and Numerics, University of M\"unster, Germany} \and Sandra May\thanks{Department of Mathematics, TU Dortmund University, Germany} \and Andreas N\"u\ss ing\footnotemark[1] \and Florian Streitb\"urger\footnotemark[2]}

\date{}

\begin{document}

\maketitle

\begin{abstract}
   We present new stabilization terms for solving the linear transport equation on a cut cell mesh 
   using the discontinuous Galerkin (DG) method in two dimensions with piecewise linear polynomials.
   The goal is to allow for explicit time stepping schemes, despite the presence of cut cells.
   Using a method of lines approach, we start with a standard upwind DG discretization for the background mesh and add penalty terms that stabilize the
   solution on small cut cells in a conservative way. Then, 
   one can use explicit time stepping, even on cut cells, with a time step length that is appropriate for the background mesh.
   In one dimension, we show monotonicity of the proposed scheme for piecewise constant polynomials and total variation diminishing in the means stability for piecewise linear polynomials. 
   We also present numerical results in one and two dimensions that support our theoretical findings.
\end{abstract}

%

\section{Introduction}
Finite element (FE) and more recently discontinuous Galerkin (DG) schemes have successfully been used on a huge variety of equations and are
 overall well understood. Therefore, research has advanced from solving model problems on simple geometries to trying to solve
real-world problems. As a result, grid generation has become a huge issue. Simulating flow around an airplane, flow in blood vessels, or phase
transitions requires to mesh very complicated geometries, which are often given as CAD models or implicitly.
The generation of corresponding body-fitted meshes is a very involved process.

In recent years, the usage of embedded boundary or cut cell meshes has become increasingly popular.
The details of these approaches vary. In this work, we will focus on the approach of cutting a given geometry out of given background mesh,
resulting in so called \emph{cut cells} along the boundary of the embedded object. These cells can have various shapes and may in particular
become arbitrarily small. Special schemes must be developed to
guarantee stability on these cells.
There already exists a significant amount of literature for stabilizing 
problems of elliptic and parabolic type
on cut cell meshes, see
e.g. \cite{bordas2018geometrically,Hansbo_Hansbo,dPLM_CAMWA_2017,cutFEM_2015,Barrett_Elliott,Bastian_Engwer,BADIA2018533,SAYE2017647,kummer2017extended}
and the references cited therein. For small cells, stability
of higher derivatives is lost and different stabilization techniques
have been proposed. A successful approach is the ghost penalty
stabilization\cite{Burman2010}, sometimes referred to as the cutFEM method~\cite{cutFEM_2015}.

For hyperbolic conservation laws on cut cell meshes, different
problems arise, compared to solving elliptic and parabolic PDEs. Probably the biggest issue is the \emph{small cell problem} -- that standard explicit
schemes are not stable on the arbitrarily small cut cells when the time step is chosen according to the cell size of the background mesh.
An additional complication is the fact that there is typically no concept of coercivity that could serve as a guideline for constructing stabilization
terms. Furthermore one wants the numerical
scheme to satisfy properties such as
monotonicity and TVD (total variation diminishing) stability in order to avoid overshoot or oscillations in the presence of discontinuities, which
could result in unphysical solutions.

In the context of \emph{finite volume schemes}, which have traditionally been used for the solution of hyperbolic conservation laws, already
several solution approaches exist that solve the small cell problem while dealing with arbitrarily small cut cells; e.g.,
the \emph{flux redistribution} method \cite{Chern_Colella,Pember_et_al,Col_Gra_Keen_Mod_JCP}, the \emph{h-box method} \cite{Berger_Leveque_AIAA,Berger_Helzel_Leveque_2005,Berger_Helzel_2012}, and the \emph{mixed explicit implicit scheme} \cite{May_Berger_explimpl}.
%
%
For the solution of hyperbolic conservation laws on cut cell meshes
using \emph{DG schemes}, very little work has been done.
As one of the first ones, in Bastian et.\,al \cite{bastian2011unfitted}
use an implicit Euler method to overcome the small cell problem
for a linear transport problem, but this approach is bound to first
order accuracy. For explicit time stepping schemes
some work relies on so called \emph{cell merging} or \emph{cell
  agglomeration} \cite{Kummer2016, Krivodonova2013}. In this approach,
cut cells that are too small are merged with neighboring cells. As
this approach puts the complexity
back into the mesh generation we do not want to consider it here.
Recently, G\"urkan and Massing \cite{Massing2018} suggested a new
scheme for solving the \emph{steady} advection-reaction problem on a cut cell mesh with potentially arbitrarily small cut cells that 
uses a ghost penalty approach. However, the authors consider only the steady case and do not discuss
properties such as monotonicity or TVD stability. Also, Sticko and
Kreiss \cite{Sticko_Kreiss} have worked on using penalty terms to
stabilize the solution of the wave equation. In both cases, the
penalty term employed has great similarity to the ghost penalty stabilization
used for elliptic problems \cite{Burman2010}.

To the best of our knowledge, the scheme suggested in this work is the first DG scheme for overcoming the small cell problem for time-dependent scalar conservation laws by dealing with the potentially arbitrarily small cut cells while ensuring monotonicity and TVD stability. Therefore, we will focus on the linear advection equation as the
standard model problem in the following. Many problems that occur for solving more complex hyperbolic equations already show up for this
simple model.

Our approach is based on stabilizing the spatial discretization.
One is free in the choice of the time stepping scheme. In particular, our stabilized spatial discretization allows for using
standard \emph{explicit} time stepping schemes everywhere, even on cut cells.
We obtain stability on cut cells by adding penalty terms. In that sense the suggested scheme is similar to the ghost penalty approach \cite{Burman2010}.
However, the details of the terms differ significantly. In particular, the terms are fundamentally different from the ones used in \cite{Massing2018,Sticko_Kreiss}. Conceptually, the terms are designed to reconstruct the proper domain of dependence
on small cut cells and their neighbors, similar to the idea of the
h-box method, but without an explicit geometry reconstruction.
In this work we consider piecewise \emph{linear} polynomials in one and two space dimensions.

This paper is structured as follows. In section \ref{sec:prelim}, we provide background information, such as triangulation and geometry information
as well as the standard scheme that we plan to stabilize. Section \ref{sec:stabilization} contains the core of this work, the
formulation of the stabilization terms in 2D. In section \ref{sec:1d-case}, we focus on the 1D case for a better understanding of the
proposed stabilization and provide theoretical results for the case of piecewise constant and piecewise linear polynomials. 
We discuss implementational aspects of the proposed scheme in section \ref{sec:impl}.
And in section \ref{sec:num-res}, we provide numerical results in both 1D and 2D that
support our theoretical findings.
We conclude with a short summary and an outlook in section \ref{sec:outlook}.

\section{Discretization}\label{sec:discretize}
\subsection{Preliminaries}\label{sec:prelim}
In this work, we focus on the linear advection equation as the classic model problem for hyperbolic scalar conservation laws, which is given by

\begin{subequations}\label{eq: hyp eq in 2d}
\begin{alignat}{2}
  u_t+\scp{\beta}{\nabla u}&=0&\quad&\text{in }\Omega \times (0,T),\\
  u &=g & &\text{on }\partial\Omega^{\text{in}} \times (0,T),\\
  u &= u_0 & &\text{on }\Omega \times \{ t = 0 \}.
\end{alignat}
\end{subequations}
Here, $\Omega \subset \mathbb{R}^2$ is a open, connected, polygonal domain,
$\partial \Omega$ denotes its boundary, and
$\partial\Omega^{\text{in}}:=\lbrace x\in\partial\Omega: \scp{\beta(x)}{n(x)}<0\rbrace$
its inflow boundary.
The velocity field $\beta \in \mathbb{R}^2$ is assumed to satisfy
$\nabla \cdot \beta =0$ and $\scp{\cdot}{\cdot}$ denotes the
canonical scalar product in $\mathbb{R}^2$.

\begin{figure}[tp]
  \centering
  \begin{tikzpicture}[scale=1.8]
    \node at (.75,1.15) {\footnotesize$\ThBg(\OmegaBg)$};
    \draw[semithick,step=0.25] (0,0) grid (1.5,1);
    \node at (1.9,0.5) {\huge$\cap$};
\begin{scope}[xshift=2.3cm]
    \node at (.75,1.15) {\footnotesize$\Omega$};
    \draw[semithick,fill=black!10!white]
    (0.0,0.0) --
    (0.3,0.0) --
    (0.5,0.03) --
    (0.75,0.132647574033) --
    (0.867352425967,0.25) --
    (0.98,0.5) --
    (1.0,0.52) --
    (1.25,0.68) --
    (1.5,0.7) --
    (1.5,1.0) --
    (0.0,1.0) --
    (0.0,0.0);
\end{scope}
    \node at (4.2,0.5) {\huge$\rightarrow$};
\begin{scope}[xshift=4.6cm]
    \node at (.75,1.15) {\footnotesize$\Th(\Omega)$};
    \draw[semithick,step=0.25] (0,0) grid (1.5,1);
    \fill[white] (0.3,0.0) --
      (0.5,0.03) --
      (0.75,0.132647574033) --
      (0.867352425967,0.25) --
      (0.98,0.5) --
      (1.0,0.52) --
      (1.25,0.68) --
      (1.5,0.7) --
      (1.51,0.7) --
      (1.51,-0.01) --
      (0.3,-0.01);
    \draw[semithick]
    (0.0,0.0) --
    (0.3,0.0) --
    (0.5,0.03) --
    (0.75,0.132647574033) --
    (0.867352425967,0.25) --
    (0.98,0.5) --
    (1.0,0.52) --
    (1.25,0.68) --
    (1.5,0.7) --
    (1.5,1.0) --
    (0.0,1.0) --
    (0.0,0.0);
\end{scope}
\end{tikzpicture}
  \caption{Construction of a cut cell mesh $\Th(\Omega)$: The background mesh $\ThBg$ of the larger domain
    $\OmegaBg$ is intersected with the computational domain $\Omega$,
    leading to cut cells $E = \widehat E \cap \Omega$, where $\widehat
    E \in \ThBg$ is an element of the background mesh.}
  \label{fig:mesh-geom-intersect}
\end{figure}
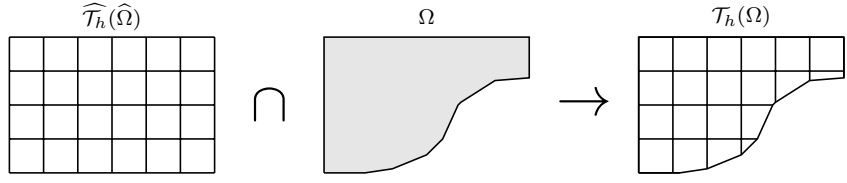

To construct our discrete approximation, we first consider a larger polygonal domain $\OmegaBg
\supset \Omega$, which is easy to mesh, see figure \ref{fig:mesh-geom-intersect}.
\begin{definition}[Cut cell mesh]
  Given a background mesh $\ThBg$ of $\OmegaBg$, we introduce a cut cell mesh $\Th(\Omega)$.
  Let $\ThBg$ be a non-overlapping set of shape-regular elements
  $\widehat E$ such that
  $\bigcup\limits_{\widehat E \in \ThBg} \overline{\widehat E} =
  \overline{\OmegaBg}$.
  For simplicity, we will assume in this paper that
  $\ThBg$ corresponds to a Cartesian background mesh but this
  is not necessary. Intersecting $\Omega$ and the background mesh
  induces the triangulation
\begin{equation*}
  \Th := \left\{~ E := \widehat E \cap \Omega~ \left| ~\widehat E \in \ThBg ~\right\}\right..
\end{equation*}
%
%
%
Note that $\Th$ is a triangulation of $\Omega$ consisting of structured (Cartesian) cells and cut cells.
%
%
%
The internal and external skeletons of the partitioning are given by
\begin{align}
  \GammaInt &=
  \left\{ e_{E_1,E_2} = \partial E_1 \cap \partial E_2
    \ \left|\  E_1, E_2 \in \Th
    \ \text{and}\  E_1 \neq E_2
    \ \text{and}\ |e_{E_1,E_2}|>0\right.\right\}. 
  \label{eq:int_skel} \\
  \Gamma_{\text{ext}} &=
  \left\{ e_{E} = \partial E \cap \partial\Omega
    \ \left|\ E \in \Th
    \ \text{and}\ |e_E|>0\right.\right\}.
    \label{eq:ext_skel}
    \end{align}
\end{definition}


\begin{definition}[Discrete Function Space]
Following the usual discontinuous Galerkin formulation we define the
discrete space $V_h^k(\Th) \subset L^2(\Omega)$ by
\begin{equation}\label{eq: def V_h}
   V_h^k(\Th) = \left\{ v_h \in L^2(\Omega)  \: \vert \: \forall E \in \Th, v_h{\vert_E} \in P^k(E) \right\},
\end{equation}
where $P^k$ denotes the polynomial space of degree $k$. 
\end{definition}

In this paper
we will only consider the cases $k=0$ and $k=1$, i.e., piecewise
constant and piecewise linear polynomials. On the skeleton $\GammaInt$,
  functions $u_h\in V^k_h$ are not well-defined. Therefore, we define jump and average as follows.

\begin{definition}[Jump \& Average]
  The \emph{jump} in normal direction over a face $e_{E_1,E_2}=\partial{E_1}\cap\partial{E_2}$ between two elements $E_1$ and $E_2$ is defined as
  \begin{align}\label{eq: jump in 2d}
    \jump{u_h}:={u_h}\vert_{_{E_1}}n_{E_1}+{u_h}\vert_{_{E_2}}n_{E_2},
  \end{align}
  where $n_{E_1},n_{E_2}\in\RR^2$ denote the unit outward normals of $E_1$ and $E_2$, respectively.
  Note that the jump of a scalar quantity is vector-valued, i.e., $\jump{u_h}\in\RR^2$.
  We define a scalar-valued jump on a face $e \in \partial E$ as
  \begin{align}\label{eq: jump scalar in 2d}
    \jp{u_h}_E:= - \jump{u_h} \cdot n_{E}.
  \end{align}
  The (scalar-valued) \emph{average} on a face is defined as
  \begin{align*}
    \average{u_h}:=\frac 1 2({u_h}\vert_{_{E_1}}+{u_h}\vert_{_{E_2}}).
  \end{align*}
\end{definition}

\subsection{Unstabilized DG Formulation}
We now introduce the DG scheme as used on structured background
cells. Additional stabilization terms will be necessary on small cut
cells, as we will discuss in section \ref{sec:stabilization}.
We will use a method of lines approach and first discretize \eqref{eq: hyp eq in 2d} in space and then in time.

\emph{The spatial discretization} of \eqref{eq: hyp eq in 2d} corresponds to the standard DG discretization using
upwind fluxes \cite{DiPietro_Ern}:
Find $u_h \in  V^k_h(\Th)$ such that
\begin{align}\label{eq: scheme 2d wo stab}
  \scp{d_tu_h(t)}{w_h}+a_h^{\text{upw}}\left(u_h(t), w_h\right)+l_h\left(w_h\right) = 0,\quad\forall w_h\in V_h^k(\Th),
\end{align}
with
\begin{align}
  &\begin{aligned}
      a_h^{\text{upw}}(u_h, w_h):=&-\sum_{E \in \Th} \int_E u_h\scp{\beta}{\nabla_hw_h}\dd{x} + \sum_{e \in \Gamma_{\text{ext}}} \int_{e} \scp{\beta}{n}^{\oplus}u_hw_h \dd{s}\\
      & +\sum_{e \in \GammaInt}\int_{e} \left( \average{u_h}\scp{\beta}{\jump{w_h}}
    + \frac{1}{2}\abso{\scp{\beta}{n_e}}\scp{\jump{u_h}}{\jump{w_h}} \right) \dd{s},
    \end{aligned}
    \label{eq:aupw}\\
    &\begin{aligned}
      l_h(w_h):=&-\sum_{e \in \Gamma_{\text{ext}}} \int_{e}\scp{\beta}{n}^\ominus g \: w_h\dd{s},\label{eq: l_h 2d}
    \end{aligned}
\end{align}
where $n\in\RR^2$ denotes the unit outer normal on $\partial\Omega$ and $n_e\in\RR^2$ denotes a unit normal on a face $e$ (of arbitrary but fixed orientation).
The negative and positive components of a quantity $x\in\RR$ are defined as
$
  x^\ominus:=\frac{\abso{x}-x}{2}$ and $x^\oplus:=\frac{\abso{x}+x}{2}.
 $ 
Note that $x^\ominus,x^\oplus\ge 0$.

The new method does not rely on a particular \emph{time stepping scheme}.
Nevertheless, it is desirable that the scheme is explicit in order
to allow for fast operator evaluation and to ease the use of
limiters. Furthermore, it should be of the same order of accuracy as
the spatial discretization and result in a discretization that is TVD.

While for piecewise constant polynomials in space the explicit Euler scheme
suffices, it will diminish the convergence order for
$V_h^1(\Th)$. We thus decided to employ the explicit second-order TVD
Runge-Kutta (RK) scheme \cite{GottliebShu} that is given for the ODE $y_t = F(y)$ by
\begin{align}\label{eq: 2nd order TVD RK}
  \begin{split}
    y^{(1)} &= y^n + \Deltat F(y^n),\\
    y^{n+1} &= \frac{1}{2} y^n + \frac{1}{2} y^{(1)} + \frac{1}{2} \Deltat F(y^{(1)}).
    \end{split}
\end{align}

A \emph{limiter} is necessary to avoid unphysical oscillations close to discontinuities when using piecewise linear polynomials.
We have chosen a Barth-Jespersen type limiter \cite{Barth_Jespersen} that has been extended to the DG setting
by exploiting the structure of the local Taylor basis \cite{Kuzmin}. This limits the gradient in such a way that the local solution of a cell $E$ evaluated at each neighboring centroid does
not over/undershoot the maximum/minimum taken over the cell's $E$ average value and the average values of all of cell's $E$ face neighbors.
The limiter is applied as a postprocessing step to both the intermediate solution $u^{(1)}$ and the solution $u^{n+1}$.

\section{Stabilization}\label{sec:stabilization}
There are two problems that need to be dealt with in order to
ensure stability on cut cells and to avoid overshoot and oscillations. First, the average mass in each cell must be controlled so that
a piecewise constant approximation does not produce oscillations. Second, the gradients must be controlled to
avoid oscillations within a single cell and unphysical evaluations at cell faces. Both goals must be reached
without violating mass conservation. For this purpose, we introduce two stabilization terms $J^0_h$ and $J^1_h$: the
first one will control average values and the second one mainly gradients.

One way to think of the stabilization terms is that they ensure that only a certain fraction of the inflow stays
in the small cut cell by transporting the remaining
part of the cut cell's inflow directly through the small cut cell into its outflow neighbors.


\begin{definition}[Capacity]
  We define the capacity of a cut cell $E$ as
  \begin{align}\label{eq:def:alphaE:omega}
    \alpha_{E,\omega}:= \min\left(\omega \frac{\abs{E}}{\Deltat \int_{\partial E}\scp{\beta}{n}^\ominus\dd{s}},1\right), \quad \omega \in \mathbb{R}, \omega \in (0,1].
  \end{align}
  For $\omega=1$, the capacity measures the fraction of the inflow
  that is allowed to flow into the cut cell $E$ without producing overshoot.
\end{definition}

We further denote with $\Ii$ the set of cut cells on which the stabilization should
be applied. The idea is that only small cut cells need stabilization:
\begin{equation}
  \label{eq:def Ii}
  \Ii  = \left\{ E \in \Th
    \left|\  E \: \text{is small cut cell and will be stabilized} \right.\right\}.
\end{equation}
We assume that for every pair of neighboring cut cells, at most one of
the two cut cells is an element of $\Ii$. While it is possible to
extend the proposed scheme to several neighboring cut cells, the
implementation would be more difficult.
We design the stabilization such that the CFL condition for explicit time stepping schemes essentially only depends on
the resolution of the background mesh
and not on the size of small cut cells from the set $\Ii$.

\begin{definition}[Inflow and outflow faces]
  For each cell $E$ we denote the set of inflow faces and outflow faces as
  \begin{align*}
    \Ff_{i}(E) & := \{e \in \partial E ~|~
                 \beta \cdot n_{E} < 0
                 \text{ on } \gamma \},\\
    \Ff_{o}(E) & := \{e \in \partial E ~|~
                 \beta \cdot n_{E} \ge 0
                 \text{ on } \gamma \},
  \end{align*}
  where $n_{E}$ denotes the unit outer normal of cell $E$ on face $e$.
\end{definition}

\begin{definition}[Outflow neighbors of $E \in \Ii$]
The set of outflow neighbors
$\NeighE$ of a cut cell $E \in \Ii$ is defined by
\begin{equation}
    \NeighE = \left\{E' \in \Th \left| \: \partial E \cap \partial E' \in \Ff_o(E), \: E' \neq E 
      \right\} \right..
\end{equation}
\end{definition}

Another way to think of 
the stabilization is that it reconstructs the proper domain of dependence, i.e., it ensures that
the outflow neighbors of a small cut cell get information from the inflow neighbors of that cut cell.
The stabilization term at a given point on an \emph{outflow} face thus
depends on the flux on the \emph{inflow} faces, measured upstream along the
trajectory.


\begin{figure}[htp]
\centering
\begin{subfigure}[t]{0.45\linewidth}
\centering
\includegraphics[width=0.7\linewidth]{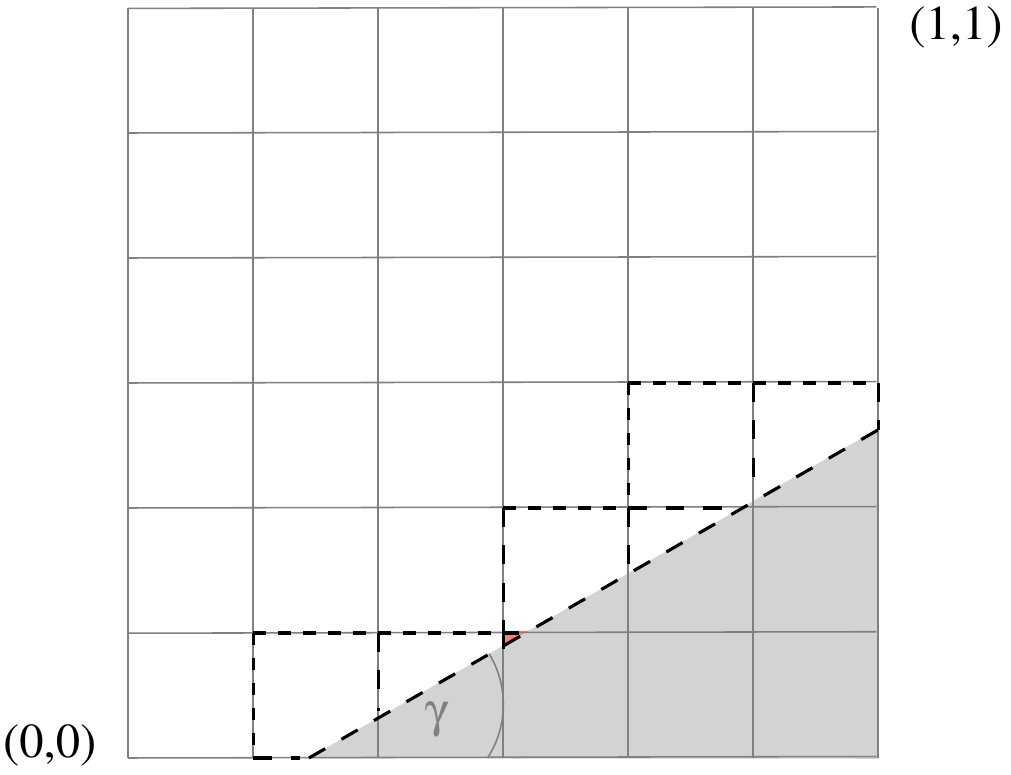}
\subcaption{Full domain}
\label{fig: traj op a}
\end{subfigure}
\hfill
\begin{subfigure}[t]{0.45\linewidth}
\centering
\includegraphics[width=0.70\linewidth]{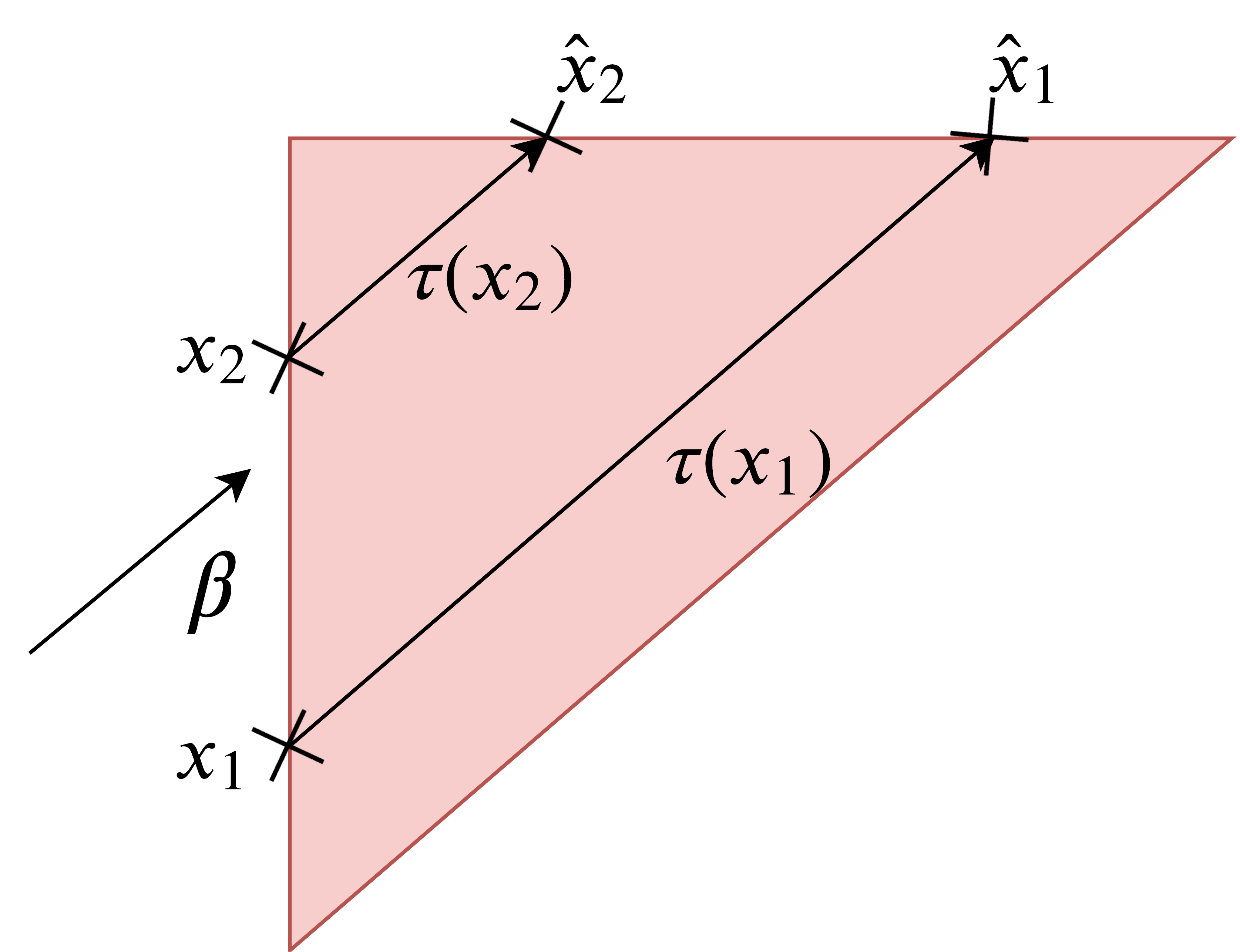}
\subcaption{Zoom on the small triangle cut cell}
\label{fig: traj op b}
\end{subfigure}%
\caption{Illustration of trajectory operator $T_E$: Consider flow parallel to a ramp. The tiny triangle cut cell (fig. \ref{fig: traj op a}) needs stabilization. Fig \ref{fig: traj op b} illustrates trajectories that start at points $x_1$ and $x_2$ on the inflow face and end at points $\hat{x}_1$ and $\hat{x}_2$ on the outflow face. The trajectory operator $T_E$ inverts this coupling:
$T_E(w)(\hat{x}_i) = w(x_i), i=1,2.$}
\label{fig: traj op}
\end{figure}

\begin{definition}[Trajectory operator]
  For every point $x_0 \in \Ff_{i}(E)$ the trajectory $\tau:\Ff_i(E)
  \to E$ describes the curve traced out by a particle starting at
  $x_0$ and being transported with $\beta$:
  \begin{align*}
    \tau(x_0) := \left\{ x \in E \left| \frac d{dt} x(t) = \beta(x(t)), x(0) = x_0, t>0 \right\}\right..
  \end{align*}

  Note that $\tau$ describes an injective mapping. We define $\tau^{-1}$ as the operator that maps
  back onto $\Ff_i(E)$ and introduce the \emph{trajectory operator}
  \begin{align*}
    T_E(w)(x) := w(\tau^{-1}(x)),
  \end{align*}
  as the evaluation of $w$ on the inflow face, following the trajectory
  backwards from $x$.
An illustration of the trajectory operator is given in figure \ref{fig: traj op}.
\end{definition}

\begin{definition}[Trajectory length]
  The trajectory length $\lTE(x)$ measures the length of a trajectory
  within the cell $E$. It is defined as
  follows: for a point $x \in E$ we identify the point $x_0$ on the
  inflow face and consider the length of the curve from $x_0$ (through
  $x$) to a point on an outflow face of $E$.
\end{definition}
In section \ref{sec:impl:trajectory} we will discuss how to realize $T_E$ and $\lTE$ practically.

The penalty term $J_h^0$ is now constructed in such a way that any inflow that exceeds the
capacity of a small cut cell is moved to the downwind cells and is
given by
\begin{eqnarray}
  \label{eq:stab:p0}
  &&\begin{split}
  J^0_h(u_h,w_h)=& 
  \sum_{E\in\Ii} 
  J^{0,E}_h(u_h,w_h),  \text{\qquad with}\\ 
  J^{0,E}_h(u_h,w_h) =&
  -\int_E \etaE \,T_E(\jp{u_h}_E)\scp{\beta}{\nabla w_h}\dd{x} \\& + 
   \sum_{e \in \Ff_{o}(E)} \int_e \etaE \,T_E(\jp{u_h}_E)
    \scp{\beta}{\jump{w_h}}\dd{s} ,
  \end{split}\\
  &&\text{and with } \etaE = 1-\alpha_{E,\frac{1}{2}}.
\end{eqnarray}

The volume contribution can be seen as a transport within the cut cell of
the quantity $\etaE = 1-\alpha_{E,\frac{1}{2}}$, which should \emph{not} remain in
the cut cell, from its inflow faces to its outflow faces. The second
term which is applied on the outflow faces transports this quantity out of the
cut cell into its downwind neighbors.  We only allow the
fraction $\alpha_{E,\frac{1}{2}}$ instead of $\alpha_{E,1}$ to stay within the
cut cell as the latter one would result in too restrictive slope
limiting, leading to close to zero gradients on the cut cells.
The main task of the second stabilization term $J^1_h$ is to restore control over
gradients.
Its structure is similar to the one of $J_h^0$,
but it employs the derivative of the discrete solution $u_h$ and uses a different scaling and a different relation between skeleton
and boundary terms: 
%
%
\begin{align}
  \label{eq:stab:p1}
  \begin{split}
  J^1_h(u_h,w_h) = &
    \sum_{E\in\Ii} 
    J^{1,E}_h(u_h,w_h), \text{\qquad with}\\
    J^{1,E}_h(u_h,w_h) =
    & - \rho \int_E \etaE \,\lTE \, T_E(\jp{\partial_\tau u_h}_E)
    \scp{\beta}{\nabla w_h}\dd{x} \\
    &+
    \sum_{e \in \Ff_{o}(E)} \int_e
    \etaE \,\lTE \, T_E(\jp{\partial_\tau u_h}_E)
    \scp{\beta}{\jump{w_h}}\dd{s} ,
  \end{split}
\end{align}
where $\partial_\tau u_h$ denotes the derivative along the trajectory,
which is evaluated 
$\partial_\tau u_h = \scp{\nabla u_h}{\beta} / \norm{\beta}$, and
$\rho = \frac{1}{2}$. (We introduce the parameter $\rho$ here as we will later examine the stability of the resulting scheme for different values of $\rho$.)
The stabilized spatial discretization is then given by:
Find $u_h \in  V^k_h(\Th)$ such that $\forall \: w_h\in V_h^k(\Th)$
\begin{equation}
\label{eq: scheme 2d with stab}
  \scp{d_t u_h(t)}{w_h}+a_h^{\text{upw}}\left(u_h(t), w_h\right)
  +J^0_h(u_h,w_h) + J^1_h(u_h,w_h) + l_h\left(w_h\right) = 0.
\end{equation}

Next, we will examine the mass conservation properties of the stabilization terms. 

\begin{definition}[Indicator function]
The indicator function of a cell or cell patch is defined as
\begin{equation}\label{eq: indicator}
    \mathbb{I}_E(x) =
    \begin{cases} 1 & \text{for } x \in E, \\
    0 & \text{otherwise} .
    \end{cases}
\end{equation}
\end{definition}

\begin{lemma} Let $E \in \Ii$.
Then
\begin{equation}
    J^{0,E}_h (u_h,\indNeighEplusE) + J^{1,E}_h (u_h,\indNeighEplusE) = 0.
\end{equation}
\end{lemma}

\begin{proof}
We first focus on $J^{0,E}_h$. 
Due to the linearity of $J^{0,E}_h$ in $w_h$, there holds 
\begin{align*}
&J^{0,E}_h(u_h,\indNeighEplusE) = J^{0,E}_h(u_h,\mathbb{I}_E) + J^{0,E}_h(u_h,\indNeighE)\\
    &=\sum_{e \in \Ff_{o}(E)} \int_{e} \etaE T_E(\jp{u_h}_E) \scp{\beta}{\mathbb{I}_E n_E} \dd{s}
    +
  \sum_{e \in \Ff_{o}(E)}  \int_e \etaE \,T_E(\jp{u_h}_E)
    \scp{\beta}{- \indNeighE n_E}\dd{s}\\
    &= 0.
\end{align*}
The same argumentation can be used for $J_h^{1,E}$.
\end{proof}

This result (together with the local mass conservation properties of the standard DG discretization \eqref{eq:aupw}) guarantees \emph{local mass conservation} in the extended control volume $E \cup \NeighE$ of a small cut cell
$E \in \Ii$: the same amount of mass that is \emph{not} allowed to stay in $E$
is redistributed to $E$'s outflow neighbors.
This slightly extended setting of local mass conservation is a natural consequence of our stabilization.

%
%

\section{The 1D case}\label{sec:1d-case}
For a better understanding of the proposed scheme and for the validation of some theoretical properties, we focus on the 1D case in this section.
WLOG, we consider the interval $I=(0,1)$ and assume $\beta>0$ to be constant.
Our PDE is given by
\begin{equation}\label{eq: 1d adv}
u_t(x,t) + \beta u_x(x,t) = 0 \text{ in } I \times (0,T), \quad u(0,t) = g(t) \text{ for } t \in (0,T),
\end{equation}
with initial data $u(x,0) =u_0(x)$.
For the analysis in this section we will focus on solving the model problem \textbf{MP} shown in figure \ref{fig:MP1}: we discretize the interval $I$ in
$N-1$ cells $E_\jj=[x_{\jj-\frac{1}{2}},x_{\jj+\frac{1}{2}}]$ of equidistant length $h$.
Then we split one cell, the cell $\kk$,
in two cells of lengths $\alpha h$ (referred to as cell $\Kone$) and $(1-\alpha) h$ (referred to as cell $\Ktwo$)
with $\alpha \in (0,\frac{1}{2}]$.
  \begin{figure}[h]
  \begin{center}
  \includegraphics[width=0.9\linewidth]{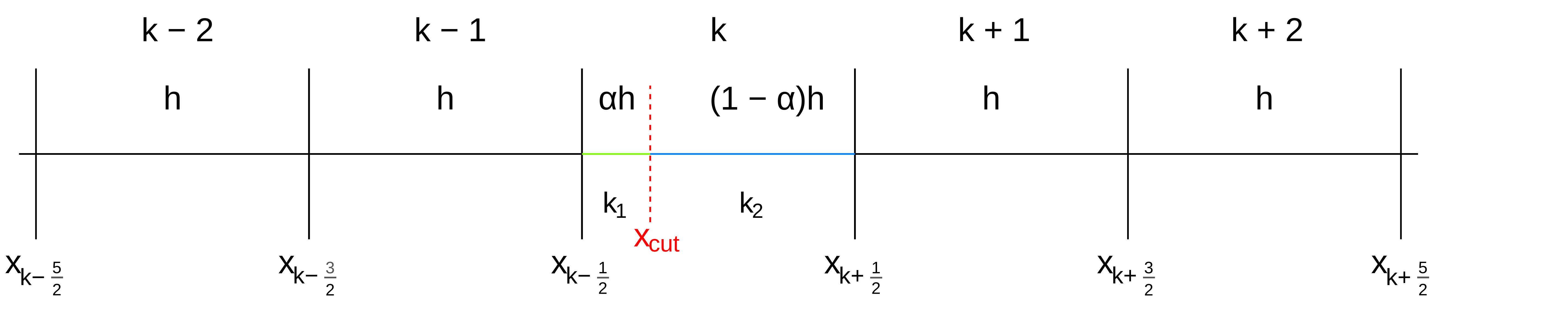}
  \caption{Model problem \textbf{MP}: equidistant mesh with cell $\kk$
    split into two cells of lengths $\alpha h$ and $(1-\alpha) h$ with
    $\alpha \in (0,\frac{1}{2}]$.
  }
  \label{fig:MP1}
  \end{center}
\end{figure}
\begin{remark}[Notation]
 Different to 2D, there is a natural order of cells in 1D. In 1D, we will therefore use the notation 
 indicated in figure \ref{fig:MP1} and refer to `cell $j$' and associate faces with $x_{j \pm 1/2}$
 instead of using `cell $E$' or `cell $E_j$' or `face $e$'.
\end{remark}

Using the notation $x_{j+\frac{1}{2}}^{\pm} = \lim_{\varepsilon \to 0} x_{j+\frac{1}{2}} \pm \varepsilon$, the
bilinear form 
\eqref{eq:aupw} simplifies to
\begin{align}\label{eq: ah_1d}
  \begin{split}
    a_h^\text{upw} (u_h,w_h) =& -\sum_{\jj=1}^N \int_{\jj} \beta u_h \partial_x w_h\dd{x}
    +\beta u_h(x_{N+\frac{1}{2}}^-)  w_h(x_{N+\frac{1}{2}}^-)\\
    &+ 	\sum_{\jj=1}^{N-1} \beta u_h(x_{\jj+\frac{1}{2}}^-)\jump{w_h}_{\jj+\frac{1}{2}}.
    \end{split}
      \end{align}
 The definitions of the jump and of $l_h$, given by \eqref{eq: jump in 2d} and \eqref{eq: l_h 2d} in 2D, reduce in 1D to
  \begin{align}\label{eq: rhs 1d}
  \jump{w_h}_{\jj+\frac{1}{2}} = w_h(\xjminus) - w_h(\xjplus), \quad l_h(w_h) :=-\beta g(0)w_h(x_{\frac{1}{2}}^+).
  \end{align}
  Further, we define the CFL number
  \begin{equation}\label{eq: lambda}
  \lambda = \frac{\beta \Deltat}{h}.
  \end{equation}
  Note that $\lambda$ is chosen only with respect to the background mesh width $h$ but
  the volume fraction $\alpha \in \left(0,\frac{1}{2}\right]$ in \textbf{MP} is allowed to become arbitrarily close to 0.
  For the considered model problem \textbf{MP}, the set of cells that need to be stabilized consists of
  $\Ii = \{ \Kone \}$ and therefore $J^0_h$ and $J^1_h$ coincide with $J^{0,k_1}_h$ and $J^{1,k_1}_h$, respectively. The stabilization terms $J_h^{0,k_1}$ and $J_h^{1,k_1}$, given by \eqref{eq:stab:p0} and \eqref{eq:stab:p1} in 2D,
  simplify significantly for the considered setup in 1D:
  \begin{itemize}
  \item For cell $\Kone$, the inflow face is $k-\frac{1}{2}$ and the outflow
  face is $k_{\text{cut}}$.
  \item For $x = x_{\Kcut}$, $T_E(\jp{u_h}_E)(x)$ in 2D simply corresponds to
  $\jump{u_h}_{\kk-\frac{1}{2}}$ in 1D: the trajectory operator $T_E$ is trivial and
  $\jp{u_h}_E$, defined in \eqref{eq: jump scalar in 2d}, on an inflow face corresponds to the
  definition of $\jump{u_h}$ in 1D, given by \eqref{eq: rhs 1d}, for $\beta > 0$.
  \item The length $\lTE$ is simply $\alpha h$.
  \item The derivative $\partial_{\tau} u_h$ in the definition of
  $J^{1,k_1}_h$ simply corresponds to $\partial_x u_h$.
  \end{itemize}
  Summarizing $J_h = J_h^0+ J_h^1$, the stabilization term in 1D for the considered setup
  is of the form
  \begin{align}\label{eq: J_h in 1d general}
    \begin{split}
  J_h(u_h,w_h) =& \beta \, \etaKone \left(\jump{u_h}_{\kk-\frac{1}{2}}+\alpha h\jump{\partial_xu_h}_{\kk-\frac{1}{2}}\right)\jump{w_h}_{\Kcut}\\
  &-\int_{{\Kone}}\beta \, \etaKone \left(\jump{u_h}_{\kk-\frac{1}{2}}+
  \alpha h \: \rho\jump{\partial_xu_h}_{\kk-\frac{1}{2}}\right)\partial_x w_h \dd{x}
  \end{split}
  \end{align}
  with
  \begin{equation}\label{eq: def eta 1d}
    \etaKone = 1 - \alpha_{\Kone,\frac{1}{2}}, \quad  \alpha_{\Kone,\frac{1}{2}} = \min \left( \frac{\alpha}{2\lambda},1 \right), \quad \rho = \frac{1}{2}.
  \end{equation}
  Note that $\alpha_{\Kone,\omega}$ is exactly the
  1D equivalent of the capacity $\alpha_{E,\omega}$ defined in \eqref{eq:def:alphaE:omega} as using
  $\lambda = \frac{\beta \Deltat}{h}$ we can reformulate
  \begin{align*}
    \alpha_{E,\omega}=\min\left(\omega\frac{\alpha h}{\beta\Deltat},1\right)
    =\min\left(\omega\frac{\alpha}{\lambda},1\right).
\end{align*}
  The resulting stabilized scheme is then given by: Find $u_h \in V_h^k(\Th)$ such that
   \begin{equation}\label{eq: stab. scheme}
     \scp{d_tu_h(t)}{w_h}+a_h^{\text{upw}}\left(u_h(t), w_h\right) +  J_h(u_h(t),w_h)
     +l_h\left(w_h\right) = 0\quad\forall w_h\in V_h^k(\Th).
     \end{equation}
\begin{definition}[\textbf{MP0} and \textbf{MP1}]
Both \textbf{MP0} and \textbf{MP1} refer to 
\begin{itemize}
    \item solving the advection equation \eqref{eq: 1d adv} using periodic boundary conditions for the model problem \textbf{MP} shown in figure
    \ref{fig:MP1},
    \item using the stabilized scheme \eqref{eq: stab. scheme} with CFL $\lambda = \frac{\beta \Deltat}{h}$.
\end{itemize}
By \textbf{MP0} we refer to the case of piecewise constant polynomials and by \textbf{MP1} to the case of 
piecewise linear polynomials, respectively.
\end{definition}

\subsection{Behavior of the space discretization \emph{without} stabilization term $J_h$}\label{sec:num-res-no-stab}
Before we discuss the properties of our stabilized scheme, we shortly present numerical results
that illuminate the different kind of instabilities that occur if one does \emph{not} use the
stabilization term $J_h$.
\subsubsection{Test 1: 1D single small cut cell} We consider the model problem \textbf{MP} and place the cell $k$ such
that $x_{k-1/2}=0.5$. We use discontinuous initial data
\begin{equation}\label{eq: disc init data}
    u_0(x) = \begin{cases} 1 & 0.1 \le x \le 0.5, \\
    0 & \text{otherwise}.
    \end{cases}
\end{equation}
We set $\beta = 1$, $\alpha = 0.001$, $\lambda = 0.5$, and $h=0.05$, and use $V_h^0(\Th)$ as well as periodic boundary conditions.
\begin{figure}[htp]
\centering
\begin{subfigure}[t]{0.45\linewidth}
\centering
\includegraphics[width=0.9\linewidth]{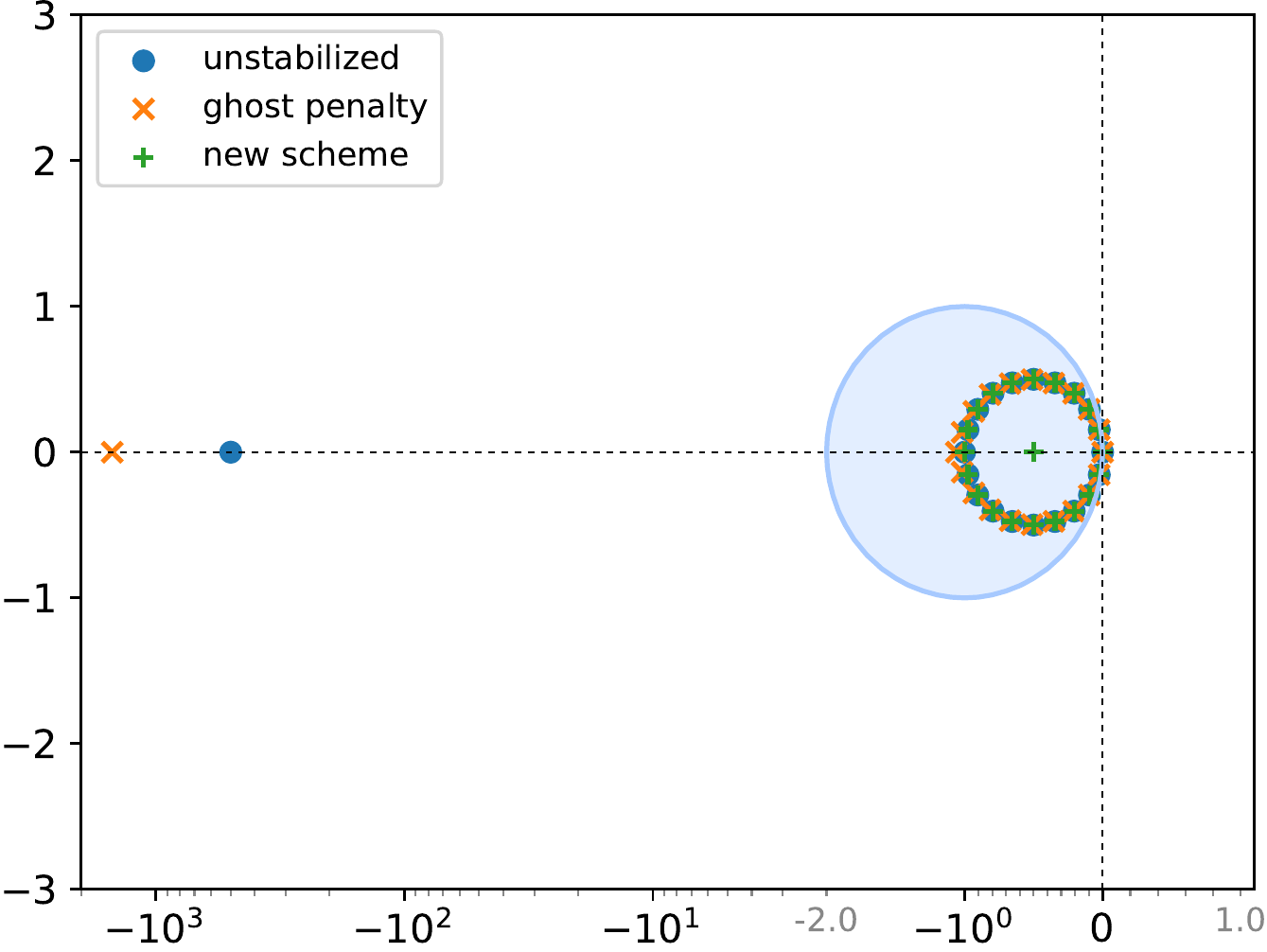}
\subcaption{Distribution of eigenvalues for the three schemes. In blue we depict the
     stability region of the explicit Euler.}
\label{subfig: unstab case 1d a}
\end{subfigure}
\hspace*{0.08\linewidth}
\begin{subfigure}[t]{0.45\linewidth}
\centering
\includegraphics[width=0.9\linewidth]{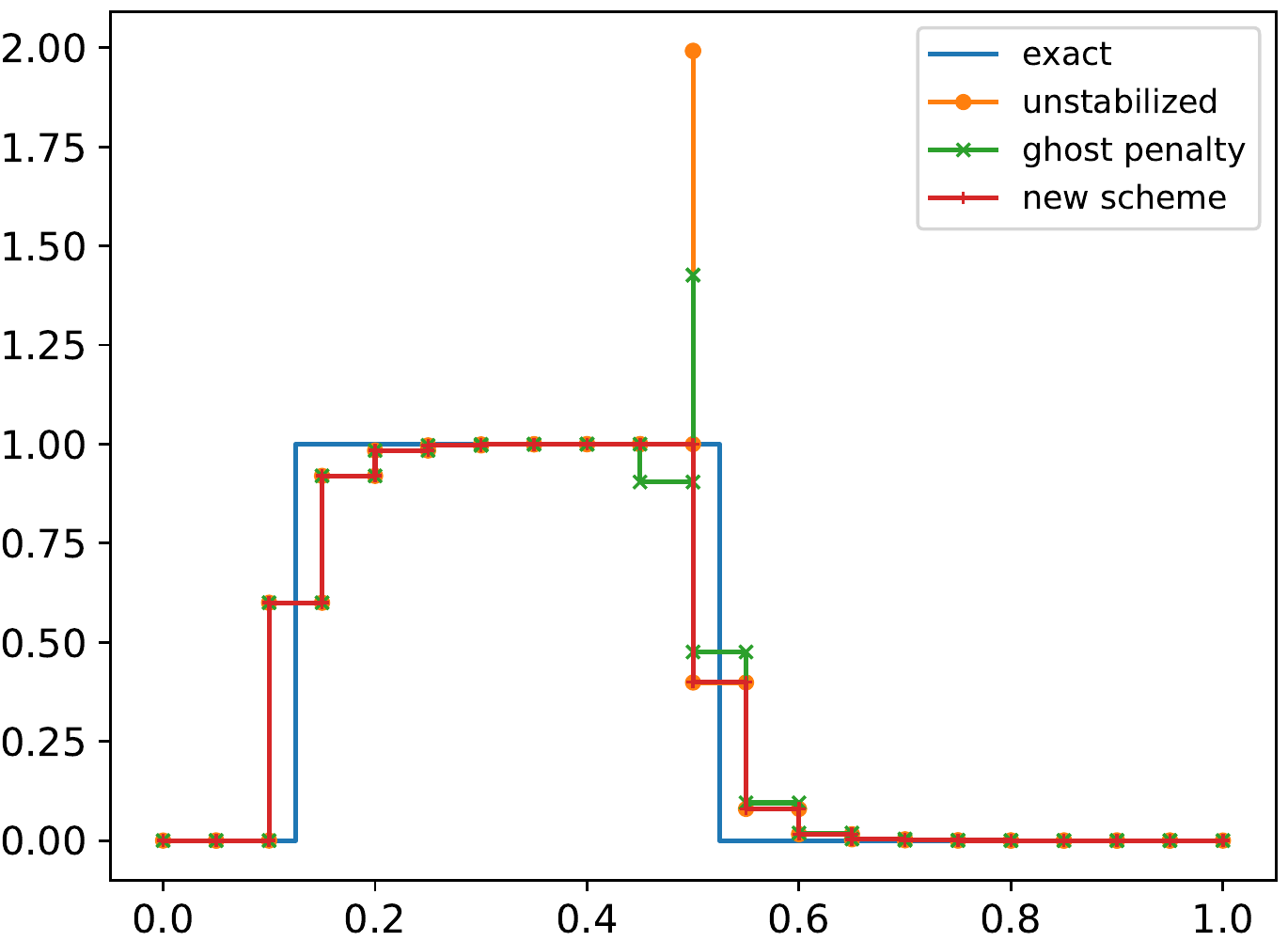}
\subcaption{Solution after one time step using implicit Trapezoidal rule in time.}
\label{subfig: unstab case 1d b}
\end{subfigure}
\caption{Numerical results for \textbf{Test 1}, comparing
  unstabilized, ghost penalty, and the proposed scheme.}
\end{figure}
%

We compare three scenarios: (1) we do not use stabilization, i.e., we apply \eqref{eq: stab. scheme} \emph{without} the term
$J_h$; (2) we apply the stabilization $J_h$; (3) instead of the $J_h$ suggested in this work, we use a straight-forward adaption of the ghost penalty stabilization \cite{Burman2010} to stabilize the problem. The ghost penalty stabilization has been used very successfully to stabilize the solution of \emph{elliptic} problems on cut cell
meshes. A first attempt to transfer the stabilization to the situation considered here would result in using a stabilization term of the form
\begin{equation}\label{eq: ghost penalty stab}
\rho_1\jump{u_h}_{\kk-\frac{1}{2}}\jump{w_h}_{\kk-\frac{1}{2}} + \rho_2\jump{u_h}_{\Kcut}\jump{w_h}_{\Kcut}
\end{equation}
(instead of the $J_h$ that we suggest).
We use $\rho_{1} = \rho_2 = \eta_{\Kone}$ (as for $J_h$).

We rewrite the described spatial discretization as a system of ODEs of the form $\underline{y}_t = L \underline{y}$ with
a suitable matrix $L$.
In figure \ref{subfig: unstab case 1d a} we present the eigenvalue distribution of $\Deltat L$ where $\Deltat = \lambda h$.
Most values lie in the stability region of the explicit Euler scheme. But if we do not stabilize or use the ghost penalty stabilization \eqref{eq: ghost penalty stab}, there is one outlier eigenvalue (corresponding to the small cut cell), leading to stability problems. With our stabilization, all values are in the stability
region of explicit Euler.

Next, we consider the usage of an implicit time stepping scheme. The result after one time step for using the implicit
Trapezoidal scheme is shown in figure \ref{subfig: unstab case 1d b}. Despite the time stepping scheme being stable in an ODE stability sense, we observe a strong overshoot: instead of staying between 0 and 1, the value on the cut cell jumps up to 2 if we do not stabilize.
When using the very simple ghost penalty stabilization \eqref{eq: ghost penalty stab} the overshoot is smaller but still significant.
The reason for this behavior is that the implicit Trapezoidal scheme is \textit{not} unconditionally TVD. In fact, no second- or higher-order scheme has this property
\cite{Spijker,Ketcheson_Macdonald_Gottlieb}. As a result, if the cell fraction $\alpha$ becomes too small, the scheme
develops unphysical oscillations and/or overshoot unless one uses a stabilization that prohibits that behavior. Our stabilization
has been designed to achieve this.

In the following, we will support these numerical results by mathematical facts: we will prove that our stabilization term $J_h$ (a) makes explicit time stepping stable again and (b) guarantees that no unphysical oscillations and/or overshoot occur. 

\subsection{Piecewise constant polynomials}
  In the special case of piecewise constant
  polynomials, the discrete solution $u_h$ on cell $j$ at time $t^n$, which we will denote by $u_j^n$, 
  is an approximation to the average of $u$ over cell $j$
   and the method is equivalent to a
  first-order finite volume scheme.
  %
  For \textbf{MP0},
  the stabilization term \eqref{eq: J_h in 1d general} reduces to (with $\etaKone$ given by \eqref{eq: def eta 1d})
 \begin{equation}\label{eq: stab P0}
	J_h(u_h,w_h)=\beta \, \etaKone\jump{u_h}_{\kk-\frac{1}{2}}\jump{w_h}_{\Kcut}.
 \end{equation}

\subsubsection{Monotonicity}
  One very desirable property of a first-order scheme for hyperbolic conservation laws is to be monotone.
  This property guarantees for example that an overshoot as seen in figure \ref{subfig: unstab case 1d b} cannot occur.
\begin{definition}[see \cite{HoldenRisebro}]
A method is called {\em monotone}, if
\begin{equation}
u^n_i\geq v^n_i \quad  \forall \ i\ \Rightarrow  \ u^{n+1}_i\geq v^{n+1}_i \quad \forall \ i .
\end{equation}
\end{definition}
Since this is challenging to verify we will use the following
equivalent definition.
\begin{definition}[see \cite{Toro}]\label{Monoton}
A method
$ u^{n+1}_\jj = H(u^n_{\jj-i_L},u^n_{\jj-i_L+1},...,u^n_{\jj+i_R}) $
is called {\em monotone}, if $\:\forall j$ there holds for every $l$ with $-i_L\le l\le i_R$
\begin{equation}\label{Def_monotone_coeff}
\frac{\partial H}{\partial u_{j+l}}(u_{j-i_L},...,u_{j+i_R})\geq 0.
\end{equation}
\end{definition}
For a linear scheme this implies that all coefficients need to be non-negative. We will verify this property for our stabilized scheme using the concept of M-matrices.
\begin{definition}[see \cite{Plemmons1977}]\label{M-Matrix_Definition}
Let $B\in\mathbb{R}^{n\times n}$ be a Z-Matrix, i.e., for $B=(b_{ij})$ there holds $b_{ij} \le 0$ for every $i\neq j$. If $b_{ii}>0$ $\forall i$ and there exists a positive diagonal matrix $D$ such that $BD$ is strictly diagonal dominant, we call B an {\em M-matrix}.
\end{definition}
\begin{lemma}[see \cite{Plemmons1977}]\label{M-Matrix_Lemma}
Let $B\in \mathbb{R}^{n\times n}$ be an M-matrix. Then $B^{-1}$ exists and $B^{-1}\geq 0.$
\end{lemma}
\begin{lemma}
Consider \textbf{MP0}.
  Discretize the time using the
  theta scheme. Then, the method is monotone under the CFL condition $\lambda < \frac{1}{2(1-\Theta)}$.
\end{lemma}
  \begin{proof}
The theta scheme for the ODE $y_t = F(y)$ is defined as
\begin{equation}
y^{n+1}=y^n+\Deltat \left[\Theta F(y^{n+1})+(1-\Theta)F(y^n)\right],
\end{equation}
with  $\Theta\in \left[0,1\right]$
($\Theta=0:$ explicit Euler, $\Theta=1$: implicit Euler,
$\Theta=\frac{1}{2}$: implicit Trapezoidal).
Let us first assume $\alpha < 2\lambda$, which is the interesting case. Then,
$\alpha_{\Kone,\frac{1}{2}} = \frac{\alpha}{2\lambda}$.
Rewriting one time step of our stabilized discretization in matrix-vector formulation results in
\begin{equation}\label{eq: 1d lin syst monotone}
  \underbrace{(M+\Delta t \Theta (A+J))}_{=:B}u^{n+1}
  = \underbrace{(M-\Delta t (1- \Theta) (A+J))}_{=:C}{} u^n.
\end{equation}
Defining $\tau_1 :=\Deltat\Theta\beta \ge 0$ and $\tau_2:=\Deltat\left(1-\Theta\right)\beta \ge 0$, the two matrices are of the form
(entries on the main diagonal are boxed)
\begin{equation*}
  B=
  \begin{small}
    \left(
    \begin{smallmatrix}
      \diagentry{h+\tau_1} & 0 & \cdots&  & &  \cdots & 0& -\tau_1\\
      -\tau_1 & \diagentry{h+\tau_1} & \ddots & & & & & 0\\
      0 & \ddots & \ddots & & & & & \vdots\\
      \vdots & & -\tau_1\frac{\alpha}{2\lambda}& \diagentry{\alpha h+\tau_1\frac{\alpha}{2\lambda}}& & & &\\
      & & -\tau_1\left(1-\frac{\alpha}{2\lambda}\right)& -\tau_1\frac{\alpha}{2\lambda} &\diagentry{(1-\alpha)h+\tau_1} & & &\vdots \\
      \vdots & & & & \ddots& \ddots& & 0\\
      0 &\cdots&&& \cdots & 0 &-\tau_1& \diagentry{h+\tau_1}\\
    \end{smallmatrix}\right)
  \end{small}
\end{equation*}
and
\begin{equation*}
  C=
  \begin{small}
    \left(
    \begin{smallmatrix}
      \diagentry{h-\tau_2} & 0 & \cdots&  & & \cdots & 0 & \tau_2\\
      \tau_2 & \diagentry{h-\tau_2} & 0 & & & & & 0\\
      0 & \ddots & \ddots & & & & &\vdots\\
      \vdots & & \tau_2\frac{\alpha}{2\lambda}& \diagentry{\alpha h-\tau_2\frac{\alpha}{2\lambda}} & & & &\\
      & & \tau_2\left(1-\frac{\alpha}{2\lambda}\right)& \tau_2\frac{\alpha}{2\lambda}& \diagentry{(1-\alpha)h-\tau_2} & & &\vdots \\
      \vdots & & & & \ddots& \ddots& & 0\\
      0 &\cdots& & &\cdots&0&\tau_2& \diagentry{h-\tau_2}\\
    \end{smallmatrix}\right)
  \end{small}
.
\end{equation*}
According to definition \ref{Def_monotone_coeff}, we need to show that $B^{-1}C$ is non-negative.
Following lemma \ref{M-Matrix_Lemma} we will verify that $B$ is an M-matrix and $C$ is non-negative.

Most cells use a standard first-order upwind scheme in space and the theta scheme with standard CFL condition in time
%
%
and obviously satisfy these conditions.
We will focus on the rows corresponding to cells $\Kone$ and $\Ktwo$, which differ from the remaining rows due to the stabilization term $J_h$. 
We start with matrix $B$ and the row corresponding to cell $\Ktwo$.
The diagonal entry $(1-\alpha) h +\tau_1$ is positive while the remaining entries are negative, since $\alpha < 2\lambda$.
To prove the diagonal dominance we compute:
\[
\lvert b_{\Ktwo,\Ktwo}\rvert-\sum_{\jj\neq \Ktwo}\lvert b_{\Ktwo,\jj}\rvert = \left(1-\alpha\right)h+\tau_1-\left(\tau_1-\frac{\alpha\tau_1}{2\lambda}\right)-\frac{\alpha\tau_1}{2\lambda}
=(1-\alpha)h>0.
\]
Similar calculations for the row corresponding to $\Kone$ imply that $B$ is an M-matrix. 

Next we will show positivity of $C$, considering
the CFL condition
$\lambda < \frac{1}{2(1-\Theta)}$. For a standard, equidistant cell,
the positivity of the entries is guaranteed already by the standard
CFL condition and thus also by our slightly stricter condition.
Straight-forward calculations further reveal that all entries in rows $\Kone$
and $\Ktwo$ are positive. %
We note that in particular the positivity of entry $c_{\Ktwo,\Ktwo}$
is guaranteed by the CFL condition and the fact that a lower
bound for the size of $\Ktwo$ is $\frac{h}{2}$. 

Finally, for $\alpha \ge 2\lambda$, there holds $\alpha_{\Kone,\frac{1}{2}} =1$ and therefore
the factor $\etaKone = 1-\alpha_{\Kone,\frac{1}{2}}$ in the definition of $J_h$ evaluates to 0, i.e.,
we would not stabilize. But as in this case $c_{\Kone,\Kone} = \alpha h - \tau_2 \ge \lambda h \Theta \ge 0$, all required properties of
$B$ and  $C$ would hold true without stabilization.
This concludes the proof.
\end{proof}

\begin{remark}[Ghost penalty stabilization]
Using the stabilization \eqref{eq: ghost penalty stab} instead of the $J_h$ given in \eqref{eq: stab P0},
this would result in additional $2 \times 2$ blocks in the matrices
$B$ and $C$ at positions $((\kk-1):\Kone,(\kk-1):\Kone)$ and
$(\Kone:\Ktwo,\Kone:\Ktwo)$, respectively. We note that it is {\em not}
possible to find factors $\rho_1$ and $\rho_2$ that guarantee that all entries of $C$ are non-negative for $\alpha \to 0$.
In contrast, our stabilizing $2 \times 2$ block has been shifted to $(\Kone:\Ktwo,(\kk-1):\Kone)$. This
shifted stabilization reflects the fact that for the hyperbolic equations there is a designated direction of information transport, whereas for elliptic problems there is not. This shifted stabilization is also a major difference to the stabilization terms used by other authors \cite{Massing2018, Sticko_Kreiss}.
  \end{remark}

\subsubsection{$L^1$ stability and TVD stability}\label{Kapitel_L1}
Next, we will show that \textbf{MP0} with {\em explicit Euler} in time is
$L^1$ stable as well as TVD stable (to be defined below) under the standard CFL condition,
independent of the size of $\alpha$.

On standard cells away from the cut cells $\Kone$ and $\Ktwo$, the scheme given by \eqref{eq: stab. scheme} 
in combination with
explicit Euler in time corresponds to the standard upwind scheme \cite{LeVeque1992}. We will assume for the rest of this subsection that $\alpha < 2\lambda$. If this was
not the case, we would still have a non-uniform mesh, but would no longer violate the CFL condition on cell $\Kone$ needed for $L^1$ and 
TVD stability and the results in the following will remain valid. 
For $\alpha < 2\lambda$,
we get the following formulae in the neighborhood of the small cut cell $\Kone$ after a minor reordering and simplification of the terms
\begin{subequations}
\begin{align}
u^{n+1}_{\kk-1}&=u^{n}_{\kk-1}-\lambda\left( u^{n}_{\kk-1}-u^{n}_{\kk-2}\right) ,\label{eq: P0_kk-1}\\
  u^{n+1}_{\Kone}&=\frac{1}{2}u^n_{\Kone} + \frac{1}{2}u^n_{\kk-1},\label{eq: P0_kk}\\
  u^{n+1}_{\Ktwo}&=\left( 1- \frac{\lambda}{1-\alpha} \right) u^{n}_{\Ktwo} + \frac{\alpha}{2(1-\alpha)}u^{n}_{\Kone} + \frac{\lambda-\frac{\alpha}{2}}{1-\alpha} u^{n}_{\kk-1}, \label{eq: P0_kk+1}\\
u^{n+1}_{\kk+1} &=u^{n}_{\kk+1}-\lambda\left( u^{n}_{\kk+1}-u^{n}_{\Ktwo}\right) .\label{eq: P0_kk+2}
\end{align}
\end{subequations}
%
%
\begin{lemma}
  Consider \textbf{MP0} with explicit Euler in time.
Then, for $\lambda<\frac{1}{2}$ there holds
\begin{align*}
\normLone{u^{n+1}}\le \normLone{u^n}\quad \forall n \ge 0.
\end{align*}
\end{lemma}
We note that the required CFL condition is {\em independent} of the size of $\alpha$
(but takes into account that the bigger cut cell $\Ktwo$ is not stabilized).
\begin{proof}
We define
\begin{equation*}
\sum_\jj \lvert u_\jj^{n+1} \rvert h_\jj = \underbrace{\sum_{\jj\le \kk-1}\lvert u_\jj^{n+1} \rvert h}_{T_1} + \underbrace{\lvert u_{\Kone}^{n+1}\rvert \alpha h}_{T_2} +  \underbrace{\lvert u_{\Ktwo}^{n+1}\rvert (1-\alpha)h}_{T_3}+  \underbrace{\sum_{\jj\geq \kk+1} \lvert u_\jj^{n+1}\rvert  h}_{T_4} .
\end{equation*}
On all cells except for cells $\Kone$ and $\Ktwo$, the standard upwind scheme is used (compare also \eqref{eq: P0_kk-1} and \eqref{eq: P0_kk+2}). Plugging in the corresponding formulae and
using $\lambda > 0$ as well as $1-\lambda > 0$ gives
\[
  T_1 \le \sum_{\jj\le \kk-2}\lvert u_\jj^n\rvert h+(1-\lambda) \lvert u_{\kk-1}^n \rvert h, \qquad
  T_4 \le \sum_{\jj\geq \kk+1}\lvert u_\jj^n \rvert h+\lambda h \lvert  u_{\Ktwo}^n\rvert.
\]
For cells $\Kone$ and $\Ktwo$, we directly obtain from equations \eqref{eq: P0_kk} and \eqref{eq: P0_kk+1} (using $1-\alpha-\lambda \ge 0$ and $2\lambda-\alpha \ge 0$)
\begin{align*}
T_2 \le \frac{\alpha h}{2}\left(\lvert u^n_{\Kone}\rvert + \lvert u_{\kk-1}^n\rvert \right),  \quad
T_3 \le \left(1-\alpha-\lambda\right)h\lvert u^n_{\Ktwo}\rvert +\frac{\alpha h }{2} \lvert u_{\Kone}^n\rvert +\left(\lambda-\frac{\alpha}{2}\right)h \lvert u^n_{\kk-1}\rvert  .
\end{align*}
Summing up the estimates for $T_1,\ldots,T_4$ implies the claim.
\end{proof}

%
\begin{definition}[\cite{Cockburn2001}]\label{Def: TVDM}
A DG scheme is called {\em TVDM} (total variation diminishing in the means) if
\begin{equation}\label{TVDM}
\text{TV}(\overline{u}^{n+1}) \le \text{TV}(\overline{u}^n) \quad\text{ with }
\text{TV}(\overline{u}^n) = \sum_\jj \lvert \overline{u}_{\jj+1}^n-\overline{u}_\jj^n\rvert
\end{equation}
holds for all $n \ge 0$. Here,  $\overline{u}_\jj^n$ denotes the mean of $u$ on cell $\jj$ at time $t_n$.
\end{definition}
For piecewise constant polynomials, the means $\overline{u}^n_\jj$ correspond to the unknowns $u^n_\jj$ and TVDM coincides with the TVD (total variation diminishing) \cite{LeVeque1992} property.
(For later use for piecewise {\em linear} polynomials we provide the more general
definition here.)

%
\begin{lemma}\label{TVD_Modell1}
   Consider \textbf{MP0} with explicit Euler in time.
Then, the scheme is TVD stable for $\lambda<\frac{1}{2}$.
\end{lemma}
\begin{proof}
We decompose
\begin{align*}
\sum_j \lvert u_{\jj}^{n+1}-u_{\jj-1}^{n+1} \rvert =& \underbrace{\sum_{\jj\le \kk-1} \lvert u_{\jj}^{n+1}-u_{\jj-1}^{n+1} \rvert }_{T_1}+\underbrace{ \lvert u_{\Kone}^{n+1}-u_{\kk-1}^{n+1} \rvert }_{T_2}+\underbrace{ \lvert u_{\Ktwo}^{n+1}-u_{\Kone}^{n+1} \rvert}_{T_3}\\
&+\underbrace{ \lvert u_{\kk+1}^{n+1}-u_{\Ktwo}^{n+1} \rvert }_{T_4}+\underbrace{\sum_{\jj\geq \kk+2} \lvert u_{\jj}^{n+1}-u_{\jj-1}^{n+1} \rvert}_{T_5}.
\end{align*}
For the standard parts of the scheme, we get
\[
  T_1\le \sum_{\jj\le \kk-2}  \lvert u_{\jj}^{n}-u_{\jj-1}^{n} \rvert +(1-\lambda) \lvert u_{\kk-1}^{n}-u_{\kk-2}^{n} \rvert, \quad
  T_5 \le \sum_{\jj\geq \kk+2} \lvert u_{\jj}^{n}-u_{\jj-1}^{n}\rvert +\lambda \lvert u_{\kk+1}^{n}-u_{\Ktwo}^{n}\rvert .
\]
Using \eqref{eq: P0_kk-1} and \eqref{eq: P0_kk}, a direct substitution of the formulae results in
\[
T_2  \le \frac{1}{2}\lvert u^n_{\Kone}-u^n_{k-1}\rvert +\lambda\lvert u^{n}_{k-1}-u^{n}_{k-2}\rvert .
\]
For $T_3$ and $T_4$, we reorder the terms resulting from the definitions \eqref{eq: P0_kk}-\eqref{eq: P0_kk+2} to get
\begin{align*}
T_3 &\le \left(1-\frac{\lambda}{1-\alpha}\right) \lvert u^n_{\Ktwo}-u^n_{\Kone} \rvert +\left(\frac{1}{2}-\frac{2\lambda-\alpha}{2(1-\alpha)}\right)\lvert u^n_{\Kone}-u^n_{\kk-1}\rvert ,\\
T_4 &\le \left(1-\lambda\right)\lvert u^n_{\kk+1}-u^n_{\Ktwo}\rvert +\frac{\lambda}{1-\alpha}\lvert u^n_{\Ktwo}-u^n_{\Kone}\rvert +\frac{2\lambda-\alpha}{2(1-\alpha)}\lvert  u^n_{\Kone}-u^n_{\kk-1}\rvert .
\end{align*}
We emphasize that all factors outside of the absolute values are non-negative due to the made assumptions.
Summing up the estimates for $T_1,\ldots,T_5$ implies the claim.
\end{proof}
\begin{remark}
  On an \emph{equidistant} mesh, a standard finite volume scheme that is monotone is automatically TVD \cite{HoldenRisebro}.
  As the standard proof for this result does not transfer to our situation, we provide here the proof for TVD stability in addition. 
\end{remark}

 \subsection{Piecewise linear polynomials}
 In this subsection we examine \textbf{MP1} more closely. 
 In this case, the stabilization term $J_h$ is given by \eqref{eq: J_h in 1d general}.
For simplicity, we will assume in this subsection that we use a
centered, rescaled moment basis, see
also section \ref{subsec: local shape fct}.
In time, we will consider the second-order explicit TVD RK scheme given by \eqref{eq: 2nd order TVD RK}.



 \subsubsection{Eigenvalue analysis}\label{Kapitel_EW}
 We now want to motivate the stability of our stabilized scheme in combination with the second-order explicit RK scheme \eqref{eq: 2nd order TVD RK}
 by means of an eigenvalue analysis using symbolic computations with \texttt{sympy}\cite{10.7717/peerj-cs.103}. We consider the model problem \textbf{MP} for the special case
 of $N=5$, i.e., we start with four equidistant cells of length $h$ and split the third cell in two cells of length $\alpha h$ and length $(1-\alpha)h$. We use
 Dirichlet boundary conditions and WLOG $\beta = 1$. 
 
 We consider the stabilized scheme \eqref{eq: stab. scheme} with the stabilization term $J_h$ given by
 \eqref{eq: J_h in 1d general}. For the construction of the stabilization term $J_h$, we have made several
 design choices, e.g., the choices of $\omega$ and $\etaKone$ and the general structure of the terms. Here, we will focus on the parameter $\rho$. We will
 show that $\rho=\frac{1}{2}$ is exactly the value that one needs to make explicit time stepping stable again.

 \begin{figure}
   \centering
   \includegraphics[height=0.31\linewidth]{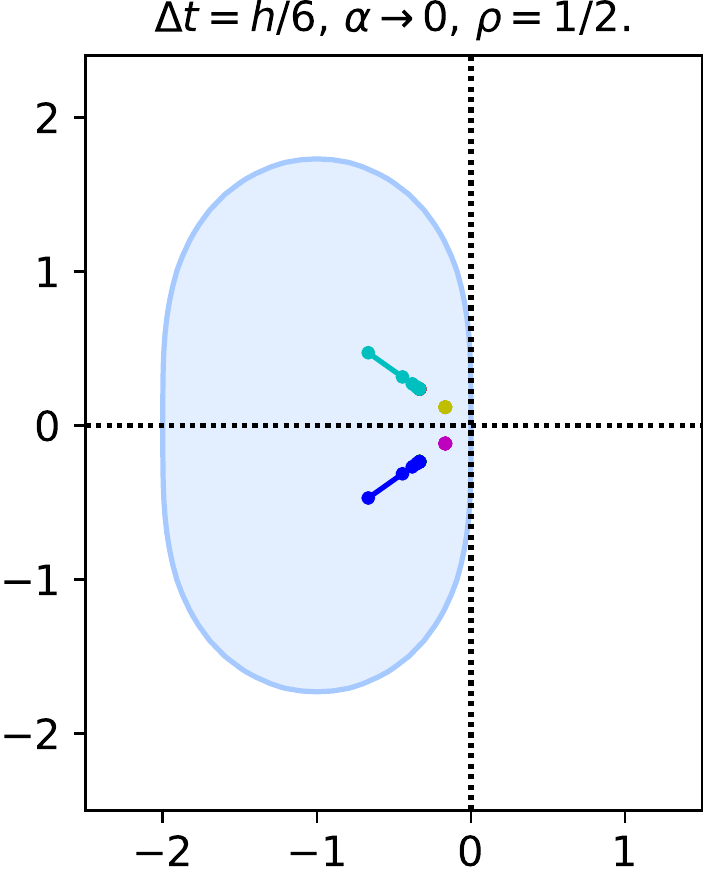}\hspace*{0.1\linewidth}
   \includegraphics[height=0.31\linewidth]{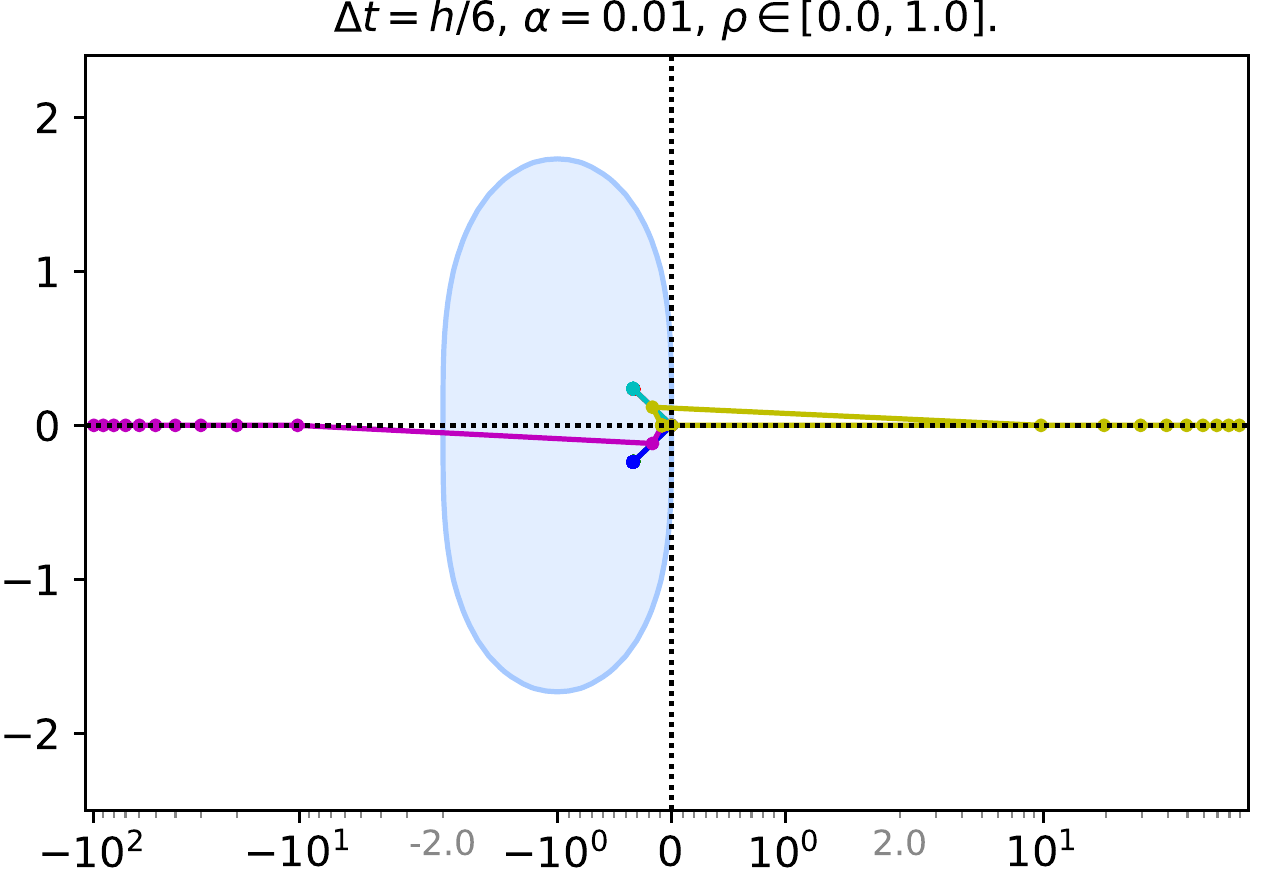}
   \caption{Eigenvalues of \eqref{eq: eigenval stab}: 
    The stability region of the explicit RK scheme is shown in blue. The colors
    refer to the different eigenvalues and show their evolution while
    changing the parameters $\alpha$ or $\rho$.
     \emph{Left:} Behavior for decreasing $\alpha$ ($\alpha = 2^{-i}, i \in
     [ 1, \ldots, 10], \rho = \frac 1 2)$. 
     \emph{Right:} Behavior for varying $ \rho$
     ($\rho \in [0,1], \alpha = 10^{-2})$.
     }
   \label{fig:evanalysis-plot}
 \end{figure}

We rewrite the variational formulation \eqref{eq: stab. scheme} as a system of ODEs with 
$M$ being the mass matrix, $A$ being the stiffness matrix incorporating $a_h^{\text{upw}}(\cdot,\cdot)$, and $J$ incorporating the corresponding
parts of $J_h(\cdot,\cdot)$.  
%
%
%
%
%
 We then symbolically setup the operator 
 \begin{equation}\label{eq: eigenval stab}
   R=M^{-1} (M - \Deltat (A+J)) - \operatorname{Id}
   ,
 \end{equation}
 which is the stability function used in the ODE stability analysis.
 As $\alpha \le \frac 1 2$, a lower bound for the size of the
 unstabilized cut cell $\Ktwo$ is $\frac h 2$. The limiting CFL number $\lambda$ for piecewise linear polynomials
 is then $\lambda=\frac{1}{6}$. We therefore choose $\Deltat = \frac{h}{6}$ to compute
 the eigenvalues of $R$. We find three pairs of complex eigenvalues:
 \begin{align*}
     \lambda_{1,2} & = - \frac{2 \pm \sqrt{2} i}{6}, \\
     \lambda_{3,4} & = \frac{2 \pm \sqrt{2} i}{6 \left(\alpha - 1\right)}, \\
     \lambda_{5,6} & = {\textstyle\frac{1 + 6 \alpha \rho - 5 \alpha -
                     2 \rho \pm \sqrt{36 \alpha^{2} \rho^{2} - 48
                     \alpha^{2} \rho + 13 \alpha^{2} - 24 \alpha
                     \rho^{2} + 28 \alpha \rho - 8 \alpha + 4 \rho^{2}
                     - 4 \rho + 1}}{2 \alpha}}.\\
   \intertext{
 The first four eigenvalues are not critical: $\lambda_{1,2}$ are
 independent of $\alpha$ and $\lambda_{3,4}$ converge to $\lambda_{1,2}$ for $\alpha \rightarrow 0$. 
 The behavior of the last two
 eigenvalues $\lambda_{5,6}$ is unclear, as they might diverge for
 $\alpha \rightarrow 0$, depending on the choice of $\rho$. We observe
 that for $\rho = \frac 1 2$ the numerator becomes independent of
 $\alpha$ resulting in}
   \lambda_{5,6} &= - \frac{2 \pm \sqrt{2} i}{2}.
 \end{align*}

 For a sequence of decreasing $\alpha$, we visualize the eigenvalues of
 $R$ and observe that for $\Deltat = \frac h  6$ all eigenvalues
 stay within the stability region of the explicit second-order TVD RK scheme
 \eqref{eq: 2nd order TVD RK} (figure
 \ref{fig:evanalysis-plot} left), indicating the stability of the
 explicit time stepping scheme.
 For fixing $\alpha=0.01$ and testing different choices of $\rho$, we observe
 (figure \ref{fig:evanalysis-plot} right) that two eigenvalues diverge.
 These observations support our
 choice of $\rho=\frac 1 2$ and are consistent with our numerical experiments, in which we never observed stability
 issues due to using an explicit time stepping scheme on tiny cut cells.

 \subsubsection{TVDM stability}\label{Kapitel_TVDM P1}
 We can only expect to obtain a TVDM stability result for piecewise linear polynomials if we apply a limiter. Different to 2D, in 1D we only reconstruct to cell
 faces instead of to neighboring centroids, which results in an MC like limiter \cite{LeVeque1992}. (In 2D, depending on the mesh, this could result in a limiter that is not linearity-preserving.) 
 By assumption, our solution $u^*_j$ on cell $j$
 resulting from solving \eqref{eq: stab. scheme} with the time stepping scheme \eqref{eq: 2nd order TVD RK},
 where $(\cdot)^*$ stands for $(\cdot)^{(1)}$ or $(\cdot)^{n+1}$, uses the centered moment basis
 \begin{equation}\label{eq: basis for limiting}
  \ubar^*_{j} +  \nabla u^*_j (x-x_{j}).
   \end{equation}
Then, we enforce
 \begin{equation}\label{eq: MC in 1d}
\nabla u^*_j = \text{minmod} \:  \left( \nabla u^*_j, \frac{\ubar^*_{j+1}-\ubar^*_{j}}{x_{j+1/2}-x_j} ,  \frac{\ubar^*_{j}-\ubar^*_{j-1}}{x_{j}-x_{j-1/2}}\right)
\end{equation}
with
\[
  \text{minmod} \: (a_1, \ldots, a_m) = \begin{cases} s \min_{1\le i \le m} \lvert a_i \rvert, & \lvert s \rvert =1, s =\frac{1}{m} \sum_{i=1}^n \text{sgn}\: (a_i) ,\\
  0, & \text{otherwise}.\end{cases}
  \]

 \begin{lemma}\label{lem: TVDM P1 Euler}
Consider \textbf{MP1} with explicit Euler in time.
Assume that the moment basis 
\eqref{eq: basis for limiting} 
and limiter \eqref{eq: MC in 1d} is used, which has been modified on cell $\kk-1$ to additionally enforce
  \begin{equation}\label{eq: cond limiting cut cell neigh text}
    \min \left(\ubar_{\kk-1}^n, \ubar_{\Kone}^n \right)
    \le u_{\kk-1}(x_{\Kcut}) \le \max \left(\ubar_{\kk-1}^n, \ubar_{\Kone}^n \right).
  \end{equation}
Then, for $\lambda<\frac{1}{4}$ the scheme is TVDM stable.
\end{lemma}
\begin{proof}
The proof follows that of  lemma \ref{TVD_Modell1}. Details are given
in appendix \ref{appx:tvdm-proof}.
\end{proof}

\begin{corollary}\label{lem: TVDM P1}
Let the assumptions of lemma \ref{lem: TVDM P1 Euler} hold true but replace the explicit Euler scheme by the
second-order scheme
   TVD RK scheme \eqref{eq: 2nd order TVD RK}. 
Then, for $\lambda<\frac{1}{4}$ the scheme is TVDM stable.
\end{corollary}

\begin{proof}
The result follows directly from the fact that TVD RK schemes are constructed to be
convex combinations of explicit Euler steps \cite{GottliebShu}.
\end{proof}

\section{Implementation}
\label{sec:impl}
The implementation is based on the DUNE \cite{dune08:1,dune08:2}
framework, a feature rich C++ finite element library. For the cut cell
discretization we rely on the \texttt{dune-udg}
package\cite{duneudg} and its integration with
\texttt{dune-pdelab}\cite{pdelab}. The \texttt{dune-udg} module
was originally developed for the
unfitted DG method\cite{Bastian_Engwer}.
We modified it to realize the 
extended
stencil of the proposed stabilization terms $J^0_h$ and $J^1_h$.

The domain is represented as a discrete level set function that uses vertex
values, i.e., as a bi-linear finite element function. Given this
representation, cut cells and their corresponding quadrature rules are
constructed using the TPMC libary\cite{tpmc}. We note that this setup results in
a restriction of the geometry that can be realized:
it only allows for a polygonal representation of the domain and for
the numerical tests we cannot resolve smooth
geometries exactly.

\subsection{Implementing the trajectory operator \& trajectory length}
\label{sec:impl:trajectory}
In 1D, the trajectory operator $T_E$ is trivial.
%
In higher space dimensions it becomes more involved to
evaluate $T_E$. We will now describe a simple approach to compute
(approximately) the evaluation of $T_E(w)(x)$ for every point $x \in E$.

Let $E \in \Ii$.
From the unsteady problem \eqref{eq: hyp eq in 2d}, we derive a
localized stationary equation given by
\begin{alignat}{2}
  \label{eq:traj_problem}
  \scp{\beta}{\nabla \tilde{w}}&=0&\quad&\text{in }E,\\
  \tilde{w} &= \scp{\beta}{w} & &\text{on }\Ff_i(E),
\end{alignat}
and natural outflow boundary conditions on $\Ff_o(E)$. The
solution $\tilde{w}$ approximates $T_E(w)(x) \approx \tilde{w}(x)$. 
To discretize \eqref{eq:traj_problem}, we employ the bilinear form
$a_h^{\text{upw}}$
 restricted to $\Omega = E$, i.e., to a single element, which we denote by $a_E^{\text{upw}}$. 
We thus search $\tilde{w}_h\in P^m(E)$ such that
\begin{align*}
  a_E^{\text{upw}}(\tilde{w}_h, v_h) = 0\quad\forall v_h\in P^m(E).
\end{align*}

\begin{remark}[Test spaces]
Depending on the shape of the streamlines of the vector field $\beta$
and the domain setup, we might choose
$m \neq k$, i.e., use a different polynomial degree for this local subproblem.
For $\beta \vert_E  \in [P^l(E)]^2$ and $w|_{e} \in P^k(e), e \in \Ff_i(E)$, the
solution $\tilde{w}$ is also polynomial and can thus be computed
exactly.
If we assume $\beta$ to depend linearly on $x$ and use $V_h^0(\Th)$ to solve \eqref{eq: hyp eq in 2d}, then
$\jp{u_h}_e \in P^0(e)$ and $\tilde{w}_h(x) = T_E(w_h)(x)$ is exact for
$\tilde{w}_h(x) \in P^1(E)$.
\end{remark}

\begin{remark}[Fast evaluation]
  Note that \eqref{eq:traj_problem} is a linear problem. We can
  interpolate $w \vert_{\Ff_i}$ into a polynomial basis and also
  $\tilde w$ is represented in a local polynomial basis. We can therefore
  \emph{precompute} 
  local (per element) trajectory operators, mapping the coefficients
  $w_i$ onto coefficient $\tilde w_j$, and store this operator as a
  small dense matrix.
\end{remark}

The \emph{trajectory length} can be computed in a very similar way. It
consists of two steps. First, we solve a local problem, which yields
the trajectory length $\lTE(x)$ for every point $x \in \Ff_i(E)$. In the second
step we compute $T_E(\lTE)$, which allows evaluating the local length
of the corresponding trajectory for every point in $\bar E$.

To compute $\lTE(x)$ we consider the following local problem:
\begin{alignat}{2}
  \label{eq:traj_len_problem}
  -|\beta|^{-1}\scp{\beta}{\nabla \tilde{l}}&=1&\quad&\text{in }E,\\
  \tilde{l} &= 0 & &\text{on }\Ff_o(E),
\end{alignat}
and natural outflow boundary conditions on $\Ff_i(E)$. It basically
computes the path integral over 1 along the trajectory and thus yields
the trajectory length $\lTE(x) = \tilde{l}(x)$. 
Again we discretize using the weak form as derived for \eqref{eq: hyp eq in 2d} and 
then we extend $\lTE(x)$ from $x \in \Ff_i(E)$ into the cell by
computing $\tilde w$ in \eqref{eq:traj_problem}.

\subsection{Construction of local shape-functions}
\label{subsec: local shape fct}
Although in principal the choice of the local basis should not have
impact on the method itself, there are practical aspects to consider.
In particular the implementation of the limiter becomes significantly
simpler, if one can separate the average value in the cell from the
local fluctuations. Thus we choose to use a moment basis. In 1D we
have chosen the local basis $\Phi_k$ in cell $k$ (in global coordinates) as
\begin{equation}
  \label{eq:basis1D}
  \{ 1, (x-x_{k})/h_k \}, \text{ with } h_k = x_{k+\frac 1 2} - x_{k-\frac 1 2}.
\end{equation}
This allows modifying the gradient without changing the average mass in the cell.

In 2D we don't have an explicit formula, but we again construct the
basis via moments in such a way, i.e., we split into a constant
function and linear functions with average zero. The linear functions are centered
around the center of mass of the cut cell, but in contrast to the 1D
case, we don't do the additional rescaling, so that the gradients are
the same as on the background mesh.

%
%

\section{Numerical results}\label{sec:num-res}
In this section we support our theoretical findings by numerical experiments in 1D and 2D.

\subsection{Numerical results for the 1D case}
For the numerical results in 1D, we solve
\eqref{eq: stab. scheme} with the stabilization term $J_h$ defined in \eqref{eq: J_h in 1d general} using $V_h^1(\Th)$.

\subsubsection{Test 2: 1D, smooth initial data}
We consider a variant of \textbf{MP1} by changing the model problem
\textbf{MP} slightly: we split all cells between $x=0.1$ and $x=0.9$ in cut cell pairs
of length $\alpha_k h$ and $(1-\alpha_k)h$ and compare 3 different scenarios:
\begin{enumerate}
    \item \textbf{S1}: the fraction $\alpha_k$ is the same for all cells, i.e., $\alpha_k = \alpha$ (`$\alpha=$');
    \item \textbf{S2}: the cell fraction $\alpha_k$ varies and is computed as $\alpha_k = 0.1X_k+10^{-6}$ with $X_k$ being a uniformly distributed random number in $(0,1)$ (`random $\alpha$');
    \item \textbf{S3}: we do not split in cut cell pairs (`equidistant').
\end{enumerate}
We use smooth initial data $u_0(x) = \sin(2\pi x)$,
set $\beta = 1$, $\lambda = \frac{1}{6}$, and use the time stepping scheme \eqref{eq: 2nd order TVD RK}.

\begin{figure}[t]
\centering
\includegraphics[width=.41\linewidth]{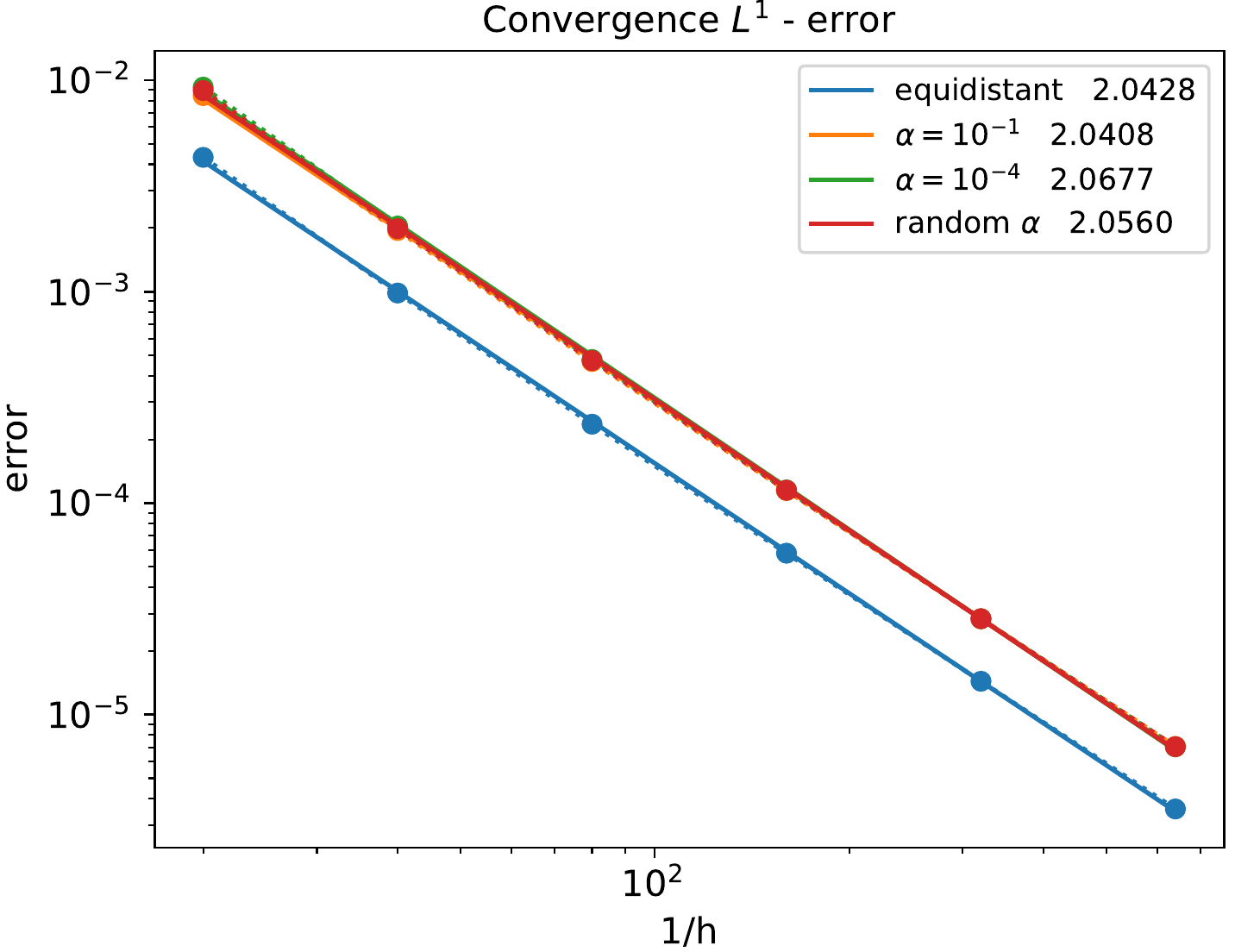}
\hspace*{.08\linewidth}
\includegraphics[width=.41\linewidth]{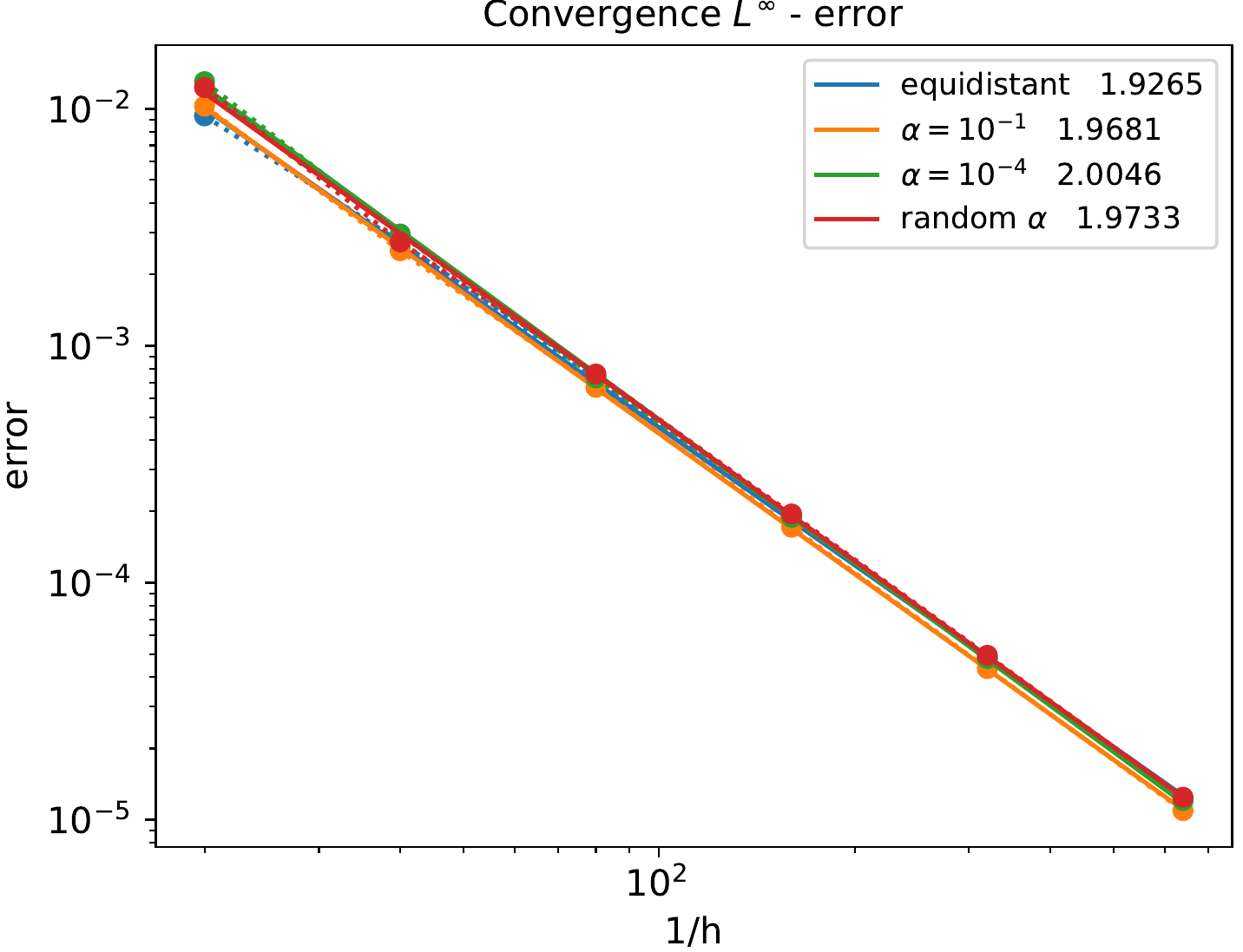}
\caption{\textbf{Test 2} (1D, smooth initial data): Optimal (second order)
  convergence rates for the error at end time $T=1$.
}
\label{fig: conv sine P1}
\end{figure}
In figure \ref{fig: conv sine P1}, the error at time $T=1$, measured
in the $L^1$ and $L^{\infty}$ norm, are shown. We observe
second-order convergence for all tested scenarios. In particular, the errors for pairs with $\alpha=0.1$ and $\alpha=10^{-4}$ are
almost the same and fairly comparable to the case of not cutting the cells.

Figure \ref{fig: 1d contour} (left) shows the solution at time $T=1$ for \textbf{S2} without and with using the
limiter described in lemma \ref{lem: TVDM P1 Euler} for a coarse mesh with $h=0.05$. Without limiter, the solution matches
very well with the exact solution. With limiter, we observe the expected peak clipping but there is no 
form of staircasing or other unwanted interaction between the limiter and the gradient stabilization.

\begin{figure}[htp]
\centering
\includegraphics[height=0.31\linewidth]{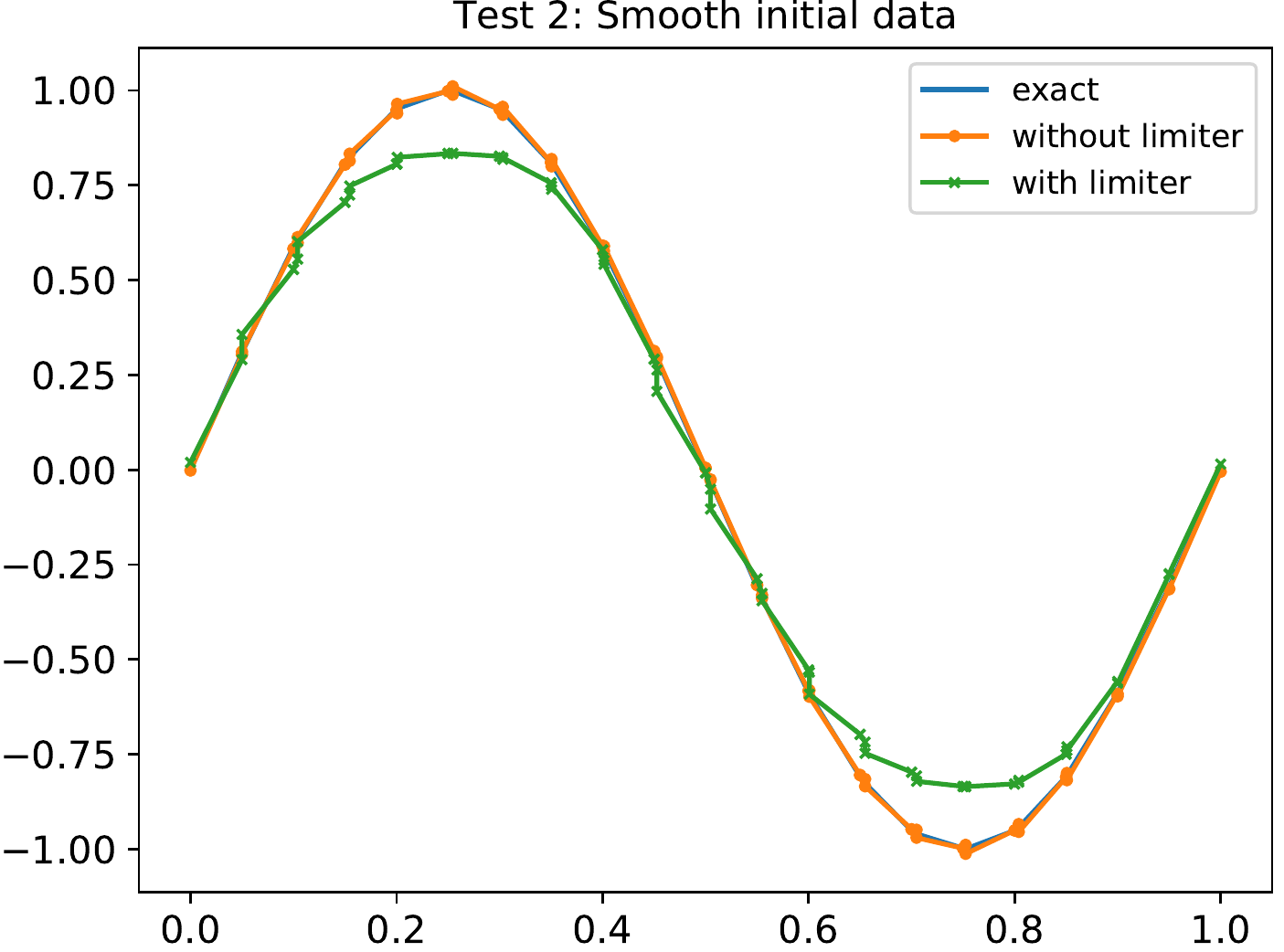}
\hspace*{0.08\linewidth}
\includegraphics[height=0.31\linewidth]{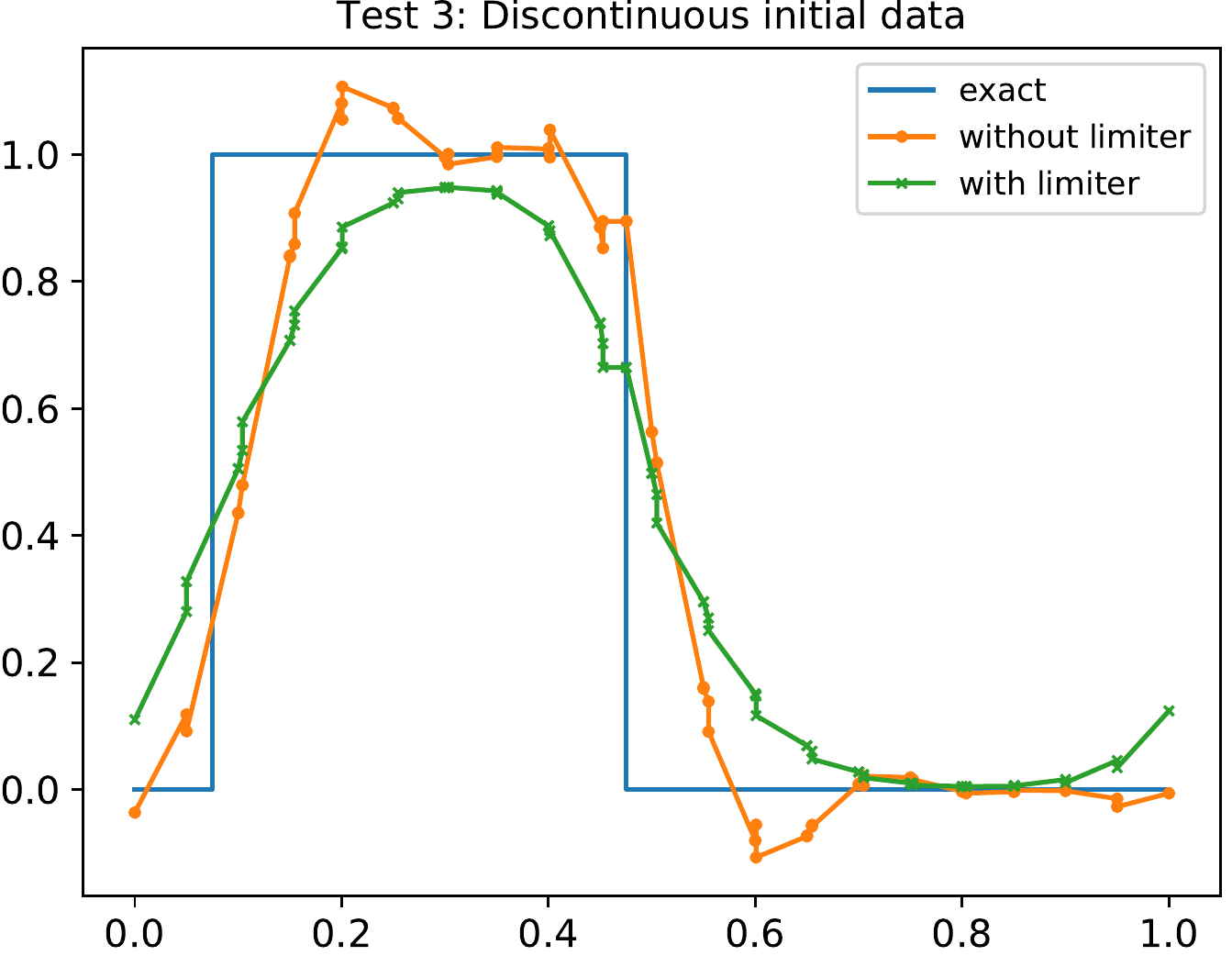}
\caption{Solutions of \textbf{Test 2} and \textbf{3} at $T=1$ for $h=0.05$ 
with and without limiter.}
\label{fig: 1d contour}
\end{figure}

\subsubsection{Test 3: 1D, discontinuous initial data}
We use the setup of \textbf{Test 2} with the scenario \textbf{S2} but use
discontinuous initial data $u_0$ given by \eqref{eq: disc init data}.
In figure \ref{fig: 1d contour} (right) we show the solution at time $T=1$ for $h=0.05$ with and without limiter.
As expected, the solution does not show overshoot when the limiter is applied. 

Additionally, we
measure $\text{TV}(\overline{u}^n)$, defined in \eqref{TVDM}, and observe that it is non-increasing during the simulation if
 the limiter is used. We note that this is not the case if we do not apply
the additional condition \eqref{eq: cond limiting cut cell neigh text}
for limiting on the inflow neighbor of a small cut cell of type $\Kone
\in \Ii$.

\subsection{Numerical results for the 2D case}
For the numerical experiments in 2D we focus on the geometry of a ramp with angle $\gamma$, see figure \ref{fig: traj op a}: we use $\OmegaBg = (0,1)^2$ and discretize it with $N\times N$ Cartesian cells, resulting in the mesh width $h = \frac{1}{N}$. Afterwards we cut out a ramp at $x=0.2001$ with angle $\gamma$ with a straight intersection, resulting in the computational domain $\Omega$. 
Our initial data $u_0$ are defined with respect to a standard Cartesian coordinate system $(x,y)$ and then transformed to our rotated and shifted coordinate system $(\hat{x},\hat{y})$ 
by using the transformation
\begin{equation}
    \begin{pmatrix} \hat{x} \\ \hat{y} \end{pmatrix} = \begin{pmatrix}\cos{\gamma}& \sin{\gamma}\\
-\sin{\gamma}&\cos{\gamma}\end{pmatrix}\cdot
\begin{pmatrix}x-0.2001\\
y\end{pmatrix}.
\end{equation}

The cut cells are located along the ramp and have various shapes and sizes. 
In our experiments, the smallest volume fractions of cut cells $E$,
computed as $\frac{\abs{E}}{h^2}$, varied typically between $10^{-5}$ and $10^{-8}$.
%
Generally the set $\Ii$ should include all cells that
are small in direction of $\beta$. In this particular setup, this
is equivalent to choosing triangular cut cells, where the cut
(i.e., the hypotenuse) is shorter than $\frac h 2$. For ease of
implementation we decided to simply use $\Ii = \{ E \in
  \Th \, | \, \frac{\abs{E}}{h^2}  < 0.1  \}$. Using this definition of $\Ii$ for $5^{\circ} \le \gamma \le 45^{\circ}$, the shortest hypotenuse of a triangle cut cell that is \emph{not} stabilized corresponds to roughly $0.6h$.

When choosing the time step size $\Deltat$, we need to take into account that the step size needs to be
appropriate for $V_h^k(\Th), k=0,1,$ for both Cartesian cells of size $h^2$ and
cut cells $E$ that are not stabilized, i.e., for cut cells $E \not\in \Ii$. Finding a tight bound 
for the latter category is non-trivial and a project in itself. We have chosen the constraint
\begin{equation}\label{eq: dt in 2d}
\Deltat \le 0.6 \frac{1}{2k+1} \frac{0.5 \: h}{\max_{ij}\lVert\beta_{ij} \rVert},
\end{equation}
which has worked well in our numerical experiments and which is in good agreement with a recently suggested CFL condition for DG schemes on
triangular meshes \cite{Krivodonova_CFL}.

For $\beta$, we choose
\begin{equation}\label{eq: vary beta NEW}
\beta^V(x,y) = \frac{1}{2\sqrt{1+\tan^2(\gamma)}}\begin{pmatrix} 1\\ \tan(\gamma)\end{pmatrix} \left( 2 + 
 (x-0.2001) \sin \gamma + y \cos(\gamma+\pi) \right).
\end{equation}
This incompressible velocity field transports the mass parallel to the ramp with decreasing speed for increasing distance to the
ramp. For comparison, we also use the constant velocity field 
\begin{equation}\label{eq: const beta NEW}
\beta^C(x,y) = \frac{2}{\sqrt{1+\tan^2(\gamma)}}\begin{pmatrix} 1\\ \tan(\gamma)\end{pmatrix}.
\end{equation}

\subsubsection{Test 4: 2D Smooth initial data}
We use the velocity field $\beta^V$,  choose smooth initial data 
\begin{equation*}
u_0(\hat{x},\hat{y})= \sin\left(\frac{\sqrt{2}\pi \hat{x}} {1-0.2001}\right),
\end{equation*}
and compute the solution at time T = 0.5 using $V_h^1(\Th)$.

\begin{figure}[t]
\centering
\includegraphics[width=.41\linewidth]{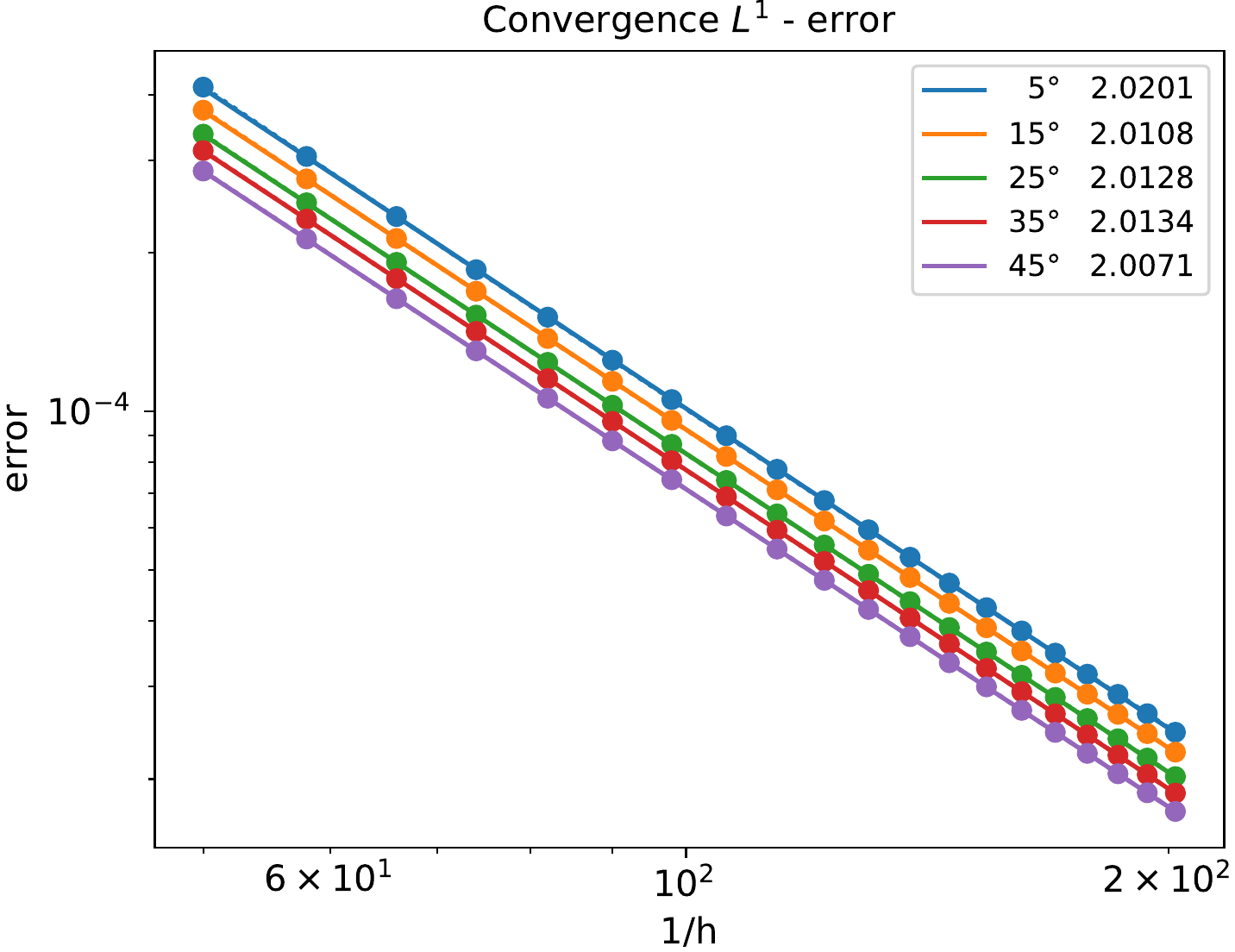}
\hspace*{.08\linewidth}
\includegraphics[width=.41\linewidth]{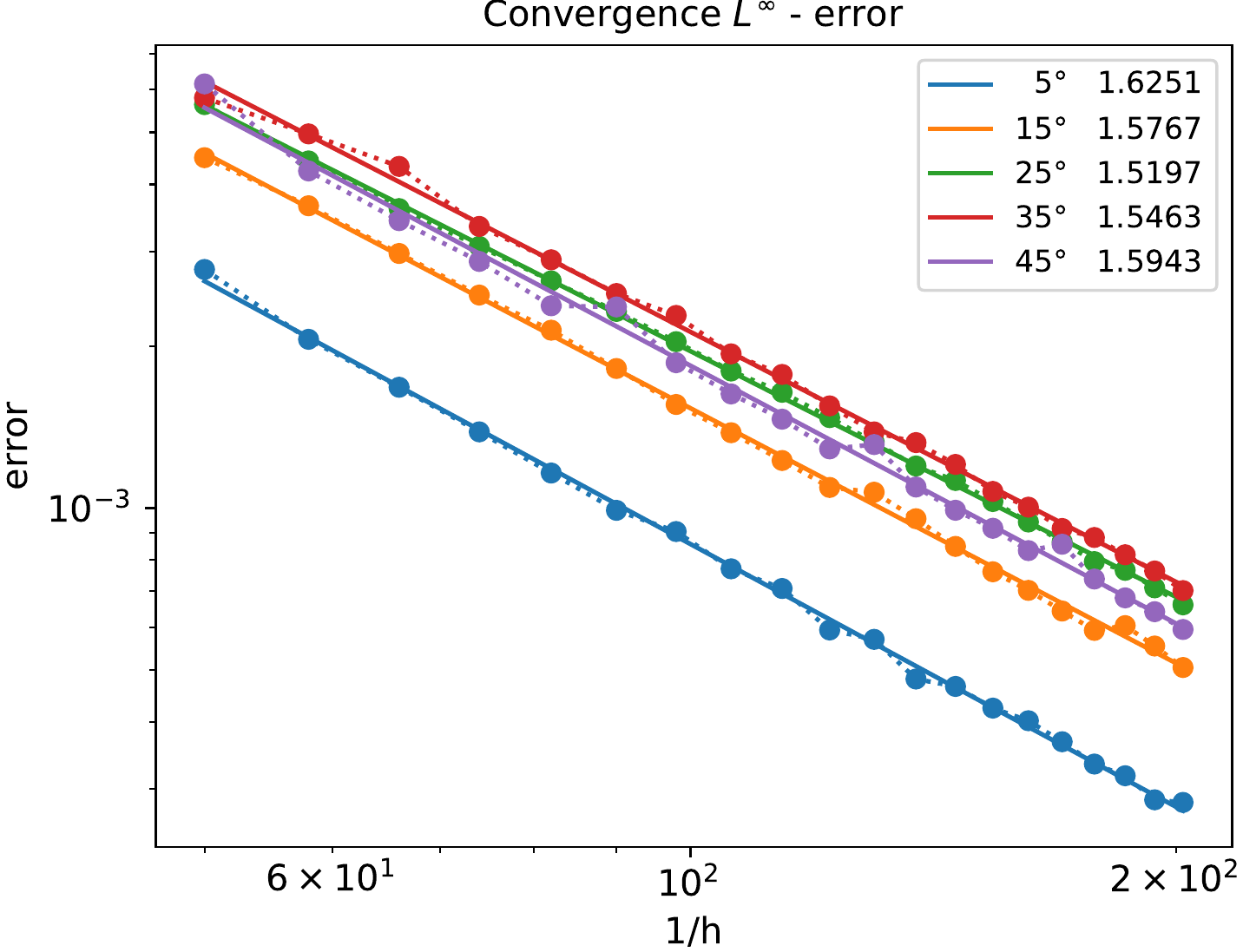}
\caption{\textbf{Test 4} (2D, smooth initial data): $L^1$ and $L^\infty$
  convergence rates for the error at end time $T$ for various angles
  $\gamma$. Reduced $L^\infty$ order, compared to the 1D tests.}
\label{fig:convergence_2D_varying NEW}
\end{figure}

Figure \ref{fig:convergence_2D_varying NEW} shows the error at time $T$ in the $L^1$ and $L^{\infty}$ norm. We observe second-order convergence in the $L^1$ norm. Due to the irregularity of the cut cells, the $L^{\infty}$ error is \emph{not smooth}, 
indicated by the zig-zag behavior of the plotted results. We therefore use a least squares fit 
to compute the convergence rates. The rates vary, depending on the angle, and lie between 1.52 and 
1.63. This slightly reduced order of convergence will need to be examined in more detail in the future. 
In general it is challenging to achieve full second-order accuracy on cut cells as the sizes of neighboring cells
differ by several orders of magnitude and therefore errors of neighboring cells do not cancel the same way as on a structured mesh.

\begin{table}[h]
\begin{center} {\small
\caption{Convergence rates ($L^1$ and $L^{\infty}$ norm) for \textbf{Test 4} for constant
  ($\beta^C$) and varying ($\beta^V$) velocity field.}
\label{Table: conv rates 2d ramp NEW}
\begin{tabular}{rcccccccccc}
  \toprule
 \multicolumn{2}{r}{Angle $\gamma$:}
  & 5$^{\circ}$ & 10$^{\circ}$ & 15$^{\circ}$ &  20$^{\circ}$ & 25$^{\circ}$ & 30$^{\circ}$ & 35$^{\circ}$ & 40$^{\circ}$ & 45$^{\circ}$\\
  \midrule
  $\beta^C$
  & $L^1$       & 2.02 & 2.01 & 2.01 & 2.00 & 2.00 & 2.01 & 2.01 & 2.01 & 2.02\\
  & $L^{\infty}$ & 1.88 & 1.68 & 1.63 & 1.60 & 1.60 & 1.56 & 1.54 & 1.53 & 1.58\\[1ex]
  $\beta^V$
  & $L^1$       & 2.02 & 2.01 & 2.01 & 2.01 & 2.01 & 2.01 & 2.01 & 2.01 & 2.00\\
  & $L^{\infty}$ & 1.62 & 1.60 & 1.57 & 1.55 & 1.51 & 1.55 & 1.54 & 1.53 & 1.59\\
\bottomrule
\end{tabular}  }   
\end{center}
\end{table}

For comparison we show in table \ref{Table: conv rates 2d ramp NEW} the convergence rates for additional
angles $\gamma$ as well as for using the constant velocity field $\beta^C$.
The results for the constant velocity $\beta^C$ are very similar to the ones for $\beta^V$, except for $5^{\circ}$ degree.
We note that it is reasonable that the results for a $5^{\circ}$ degree ramp are better as most cut cells have full length in flow direction, i.e., such a ramp contains significantly fewer `problematic' cut cells than a ramp with a higher angle and the sizes of neighboring cut cells do not differ as strongly. 

Finally, in figure \ref{fig: 2d contour var beta} (left column), we show the solution 
for $\beta^V$ and $\gamma=30^\circ$ on a coarse mesh of $N=30$.
The 1D profile along the cut boundary shows the expected sine curve.
For the contour plot, we observe that the contour lines are straight lines all the way to the boundary.
We note that for the initial data $u_0$ the contour lines are perpendicular to the ramp. Due to $\beta^V$ having decreasing speed for increasing distance to the ramp, the lines have been rotated during the simulation.

\begin{figure}[bht]
\centering
  \begin{tabular}{*{4}{p{0.21\linewidth}}}
\includegraphics[width=\linewidth]{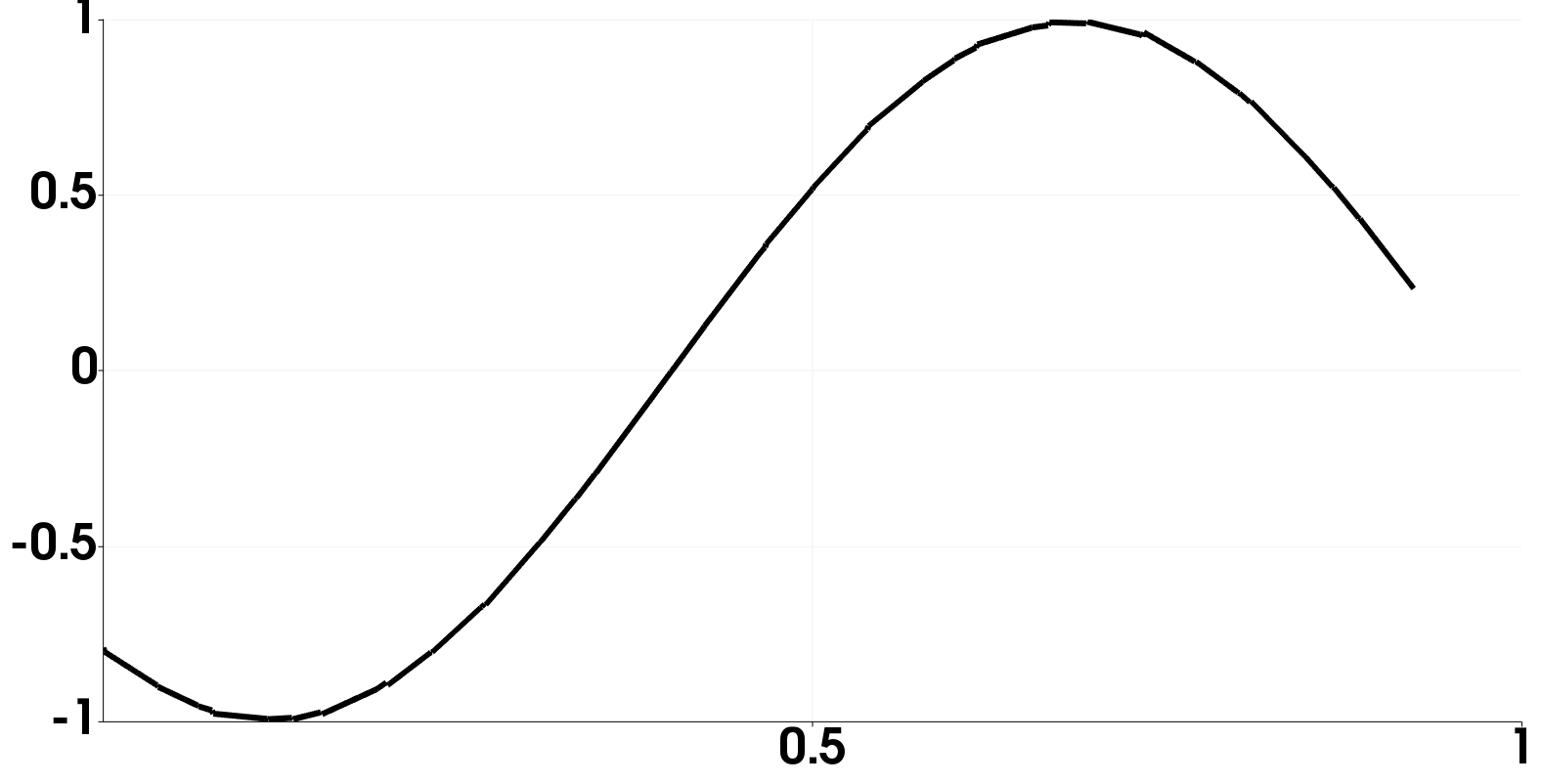}
&
\includegraphics[width=\linewidth]{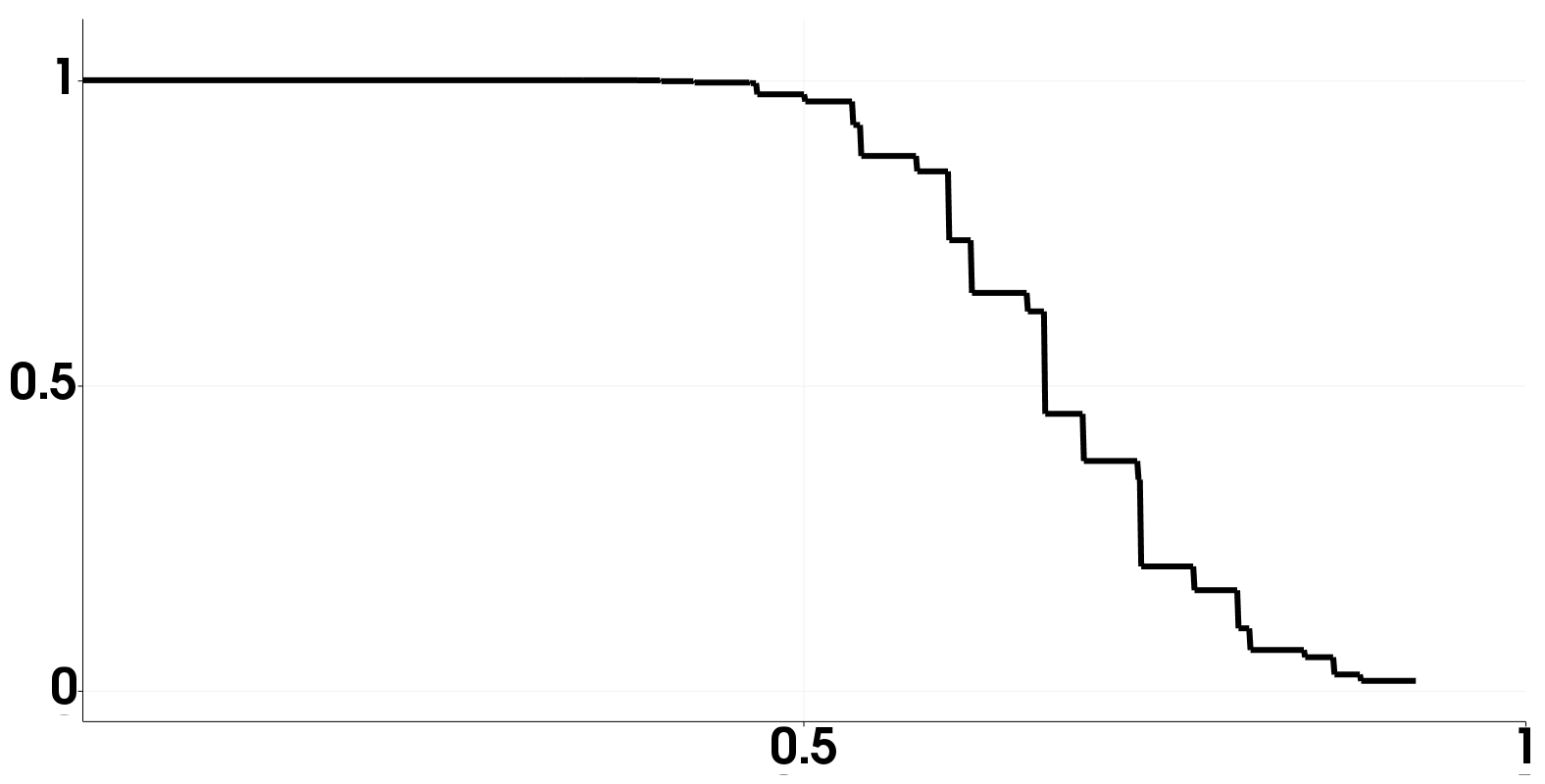}
&
\includegraphics[width=\linewidth]{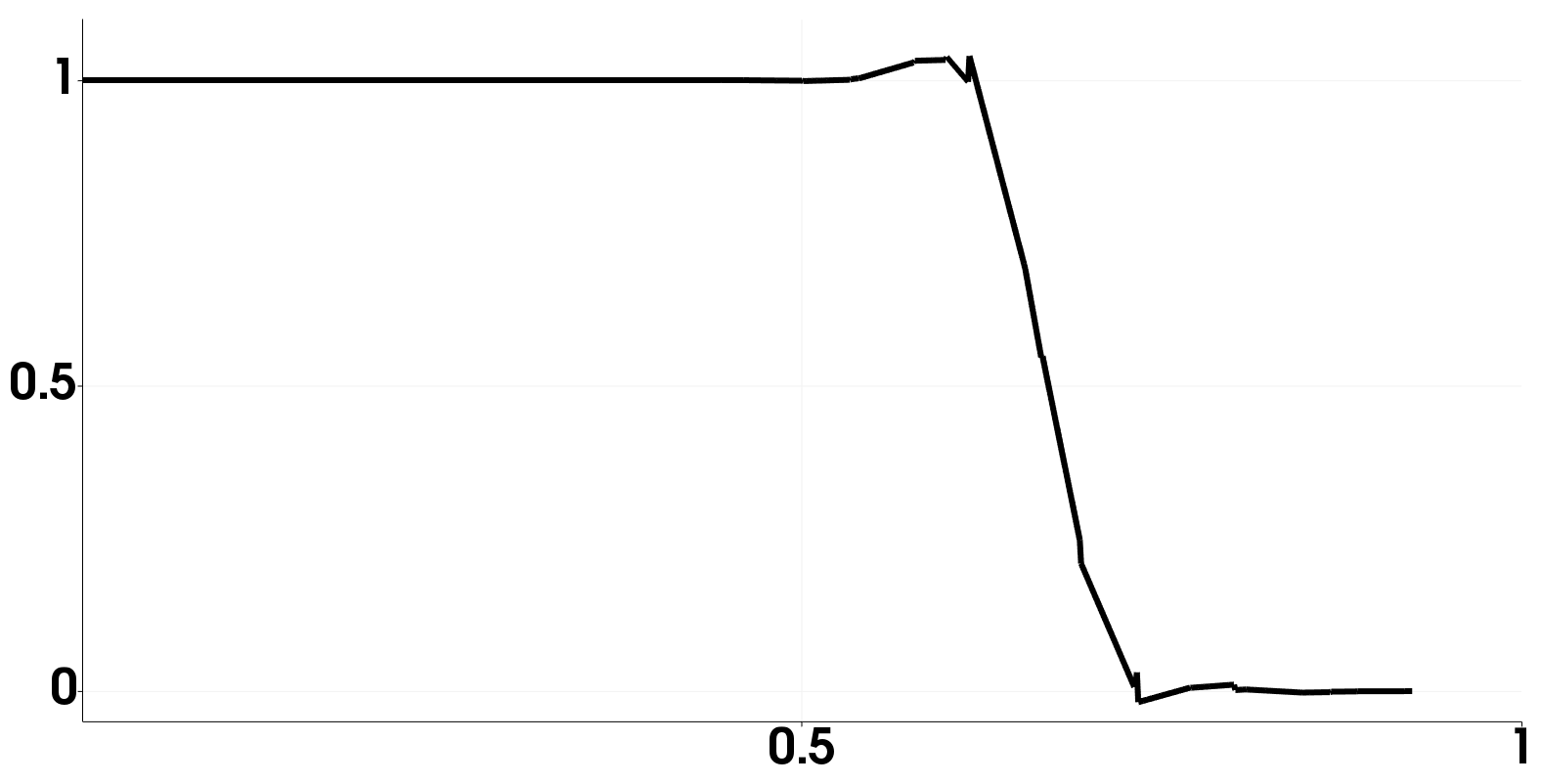}
&
\includegraphics[width=\linewidth]{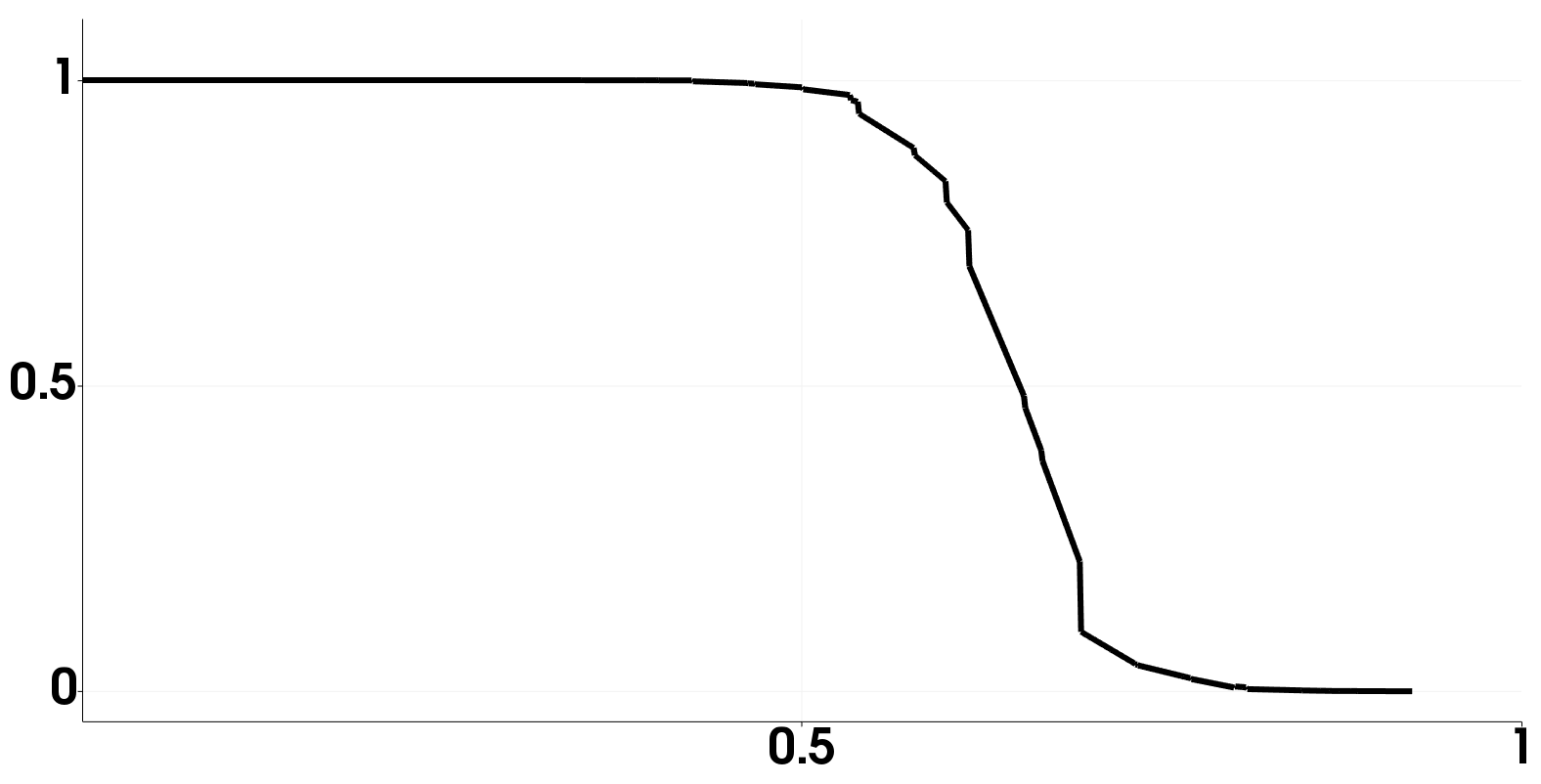}
\\
\includegraphics[width=\linewidth]{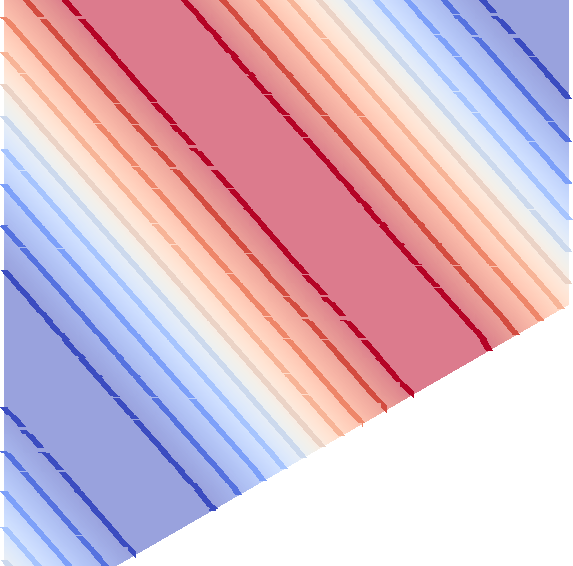}
&
\includegraphics[width=\linewidth]{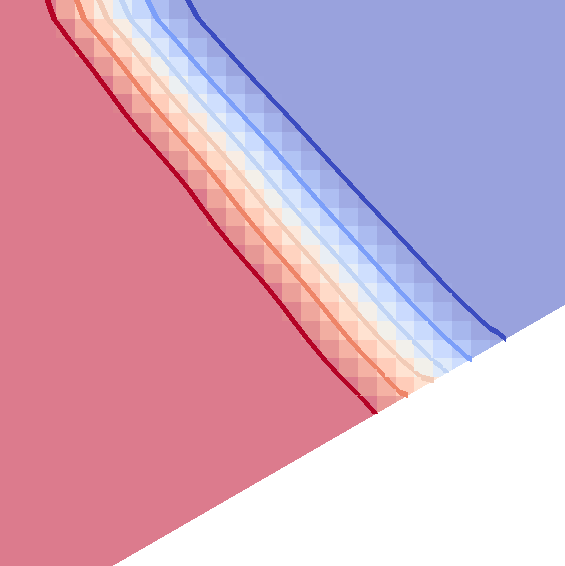}
&
\includegraphics[width=\linewidth]{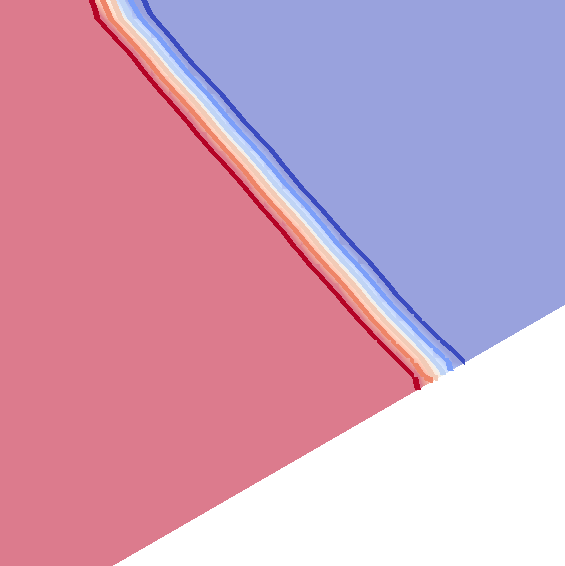}
&
\includegraphics[width=\linewidth]{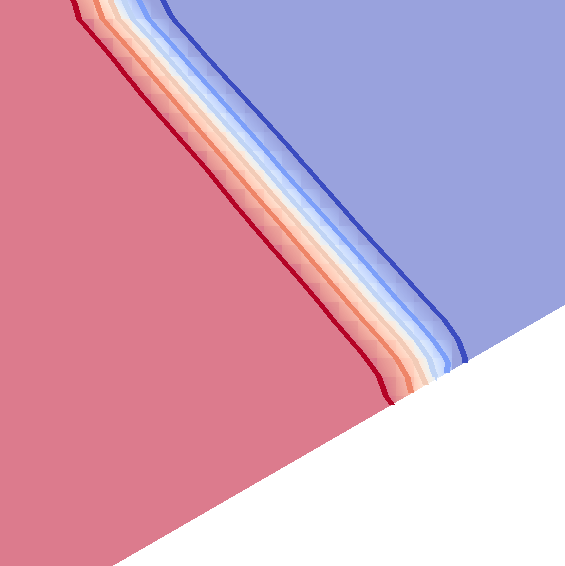}\\
{\small $V_h^1(\Th)$ w/o limiter}
&{\small $V_h^0(\Th)$}
&{\small $V_h^1(\Th)$ w/o limiter}
&{\small $V_h^1(\Th)$ with limiter}
\end{tabular}
\caption{Solutions of \textbf{Test 4} (column 1) and \textbf{Test 5}
  (columns 2-4) at time $T=0.4$, with
  $N=30$ for four different scenarios.
  Straight contour lines in columns 2 and 3 indicate that our stabilization does not add extra diffusion on the cut cells.
  The upper row plots show $u$
  along the cut boundary.}

\label{fig: 2d contour var beta}
\end{figure}

\subsubsection{Test 5: 2D Discontinuous initial data}
We use the velocity field $\beta^V$ and compute until time $T=0.4$ for discontinuous initial data
\begin{equation*}
    u_0(\hat{x},\hat{y}) = \begin{cases}
        1 & \text{ for } \hat{x} < \frac{4}{15}, \\
        0 & \text{ otherwise.} 
    \end{cases}
\end{equation*}

Figure \ref{fig: 2d contour var beta} (columns 2-4) shows the discrete solution 
for $N=30$ for $\gamma=30^\circ$ for both $V_h^0(\Th)$ and $V_h^1(\Th)$,
the latter without and with limiter. For both $V_h^0(\Th)$ and $V_h^1(\Th)$ with limiter, the discrete
solution stays between 0 and 1 and we do not observe overshoot. This is also indicated by the 1D line plots along the
cut boundary. 
Further, we verified numerically for
$V_h^0(\Th)$ that all matrix indices in the matrix $B^{-1}C$, compare
\eqref{eq: 1d lin syst monotone} for the 1D version, are
non-negative. 

Examining the contour lines more closely, we find that for $V_h^0(\Th)$ and $V_h^1(\Th)$ without limiter
the contour lines are straight lines all the way to the boundary, indicating that our stabilization does not add extra diffusion on the cut cells.
For the case of $V_h^1(\Th)$ with limiter, we observe slightly more diffusion along the boundary. We attribute this to the limiter: on a Cartesian cell, we only limit in $x$- and $y$-direction. On cut cells however we also limit roughly in advection direction when reconstructing to neighboring centroids. We expect this behavior to improve if a more accurate limiter, e.g., the LP limiter \cite{May_Berger_LP} is used, but this is not the focus of this work.

\section{Conclusions and Outlook}\label{sec:outlook}

We have presented a new stabilization for DG schemes for
linear scalar conservation laws on cut cell meshes that solves the small cell problem and makes explicit time
stepping stable again. 
Our stabilization is designed to only let a certain portion of the inflow of a small cut cell stay in that small cut cell and to transport the remaining
portion directly into the cut cell's outflow neighbors. As a by-product, we reconstruct the proper
domain of dependence of the small cut cell's outflow
neighbors. In that sense our stabilization relies on similar ideas as the $h$-box method \cite{Berger_Helzel_Leveque_2005,Berger_Helzel_2012}, but without an explicit geometry reconstruction.

The approach for realizing these ideas in a DG setting 
was inspired by the ghost penalty method \cite{Burman2010} but significant changes were necessary to
adjust the terms that were developed for elliptic problems to the setting of hyperbolic equations.
In this contribution, we have focused on using two standard
explicit time stepping schemes (the explicit Euler
scheme and the standard second-order TVD RK scheme)
and the standard Barth-Jespersen limiter but 
there are no explicit obstacles for working
with other
  choices of time stepping schemes and limiters.

Our stabilization ensures conservation (in a slightly extended meaning). In one dimension, we have shown that
the stabilized scheme is monotone for piecewise constant polynomials and total variation diminishing in the means for
piecewise linear polynomials in combination with explicit time
stepping schemes.
In numerical tests we observed optimal convergence rates in the $L^1$
  and $L^\infty$ norm for 1D. In 2D, the numerical results in the
  $L^1$ norm showed again full second-order convergence but the
  convergence rate in the $L^{\infty}$ norm was 
  slightly
  reduced. This will be examined in more detail
  in future work.
  We also plan to extend the stabilization to higher-order polynomial degrees and to non-linear conservation laws.\\

\appendix

\section{Proof of TVDM property}\label{appx:tvdm-proof}
\begin{proof}[Proof of lemma \ref{lem: TVDM P1 Euler}]
The structure follows that of the proof of lemma \ref{TVD_Modell1}. 
We will show that%
\[\textstyle
\sum_\jj \lvert \ubar_j^{n+1}-\ubar_{j-1}^{n+1} \rvert \leq \sum_\jj
\lvert \ubar_j^n-\ubar_{j-1}^n \rvert
\qquad  \forall n \ge 0 \text{ and } \forall j.
\]
We exploit the fact that a moment basis is used.
Testing with
$w_h = \mathbb{I}_j$ yields the update for $\ubar_j$. For the update of the gradient, it is sufficient to consider the information that the postprocessing by
applying the limiter provides.
Again it is sufficient to consider the case $\alpha < 2\lambda$. 

In the following $u_j(x)$ denotes the full (linear) solution of cell
$j$ evaluated at $x$, which can be written as $u_j(x) = \ubar_j + (x-x_j)\nabla u_{j}$.
We first provide the update formulae in the neighborhood of the cut cell $\Kone$ induced by the stabilization \eqref{eq: stab. scheme}:
\begin{subequations}
\begin{align}
\ubar_{k-1}^{n+1} &= \ubar_{k-1}^n-\lambda\left(u_{k-1}^n(x_{k-\frac{1}{2}})-u^n_{k-2}(x_{k-\frac{3}{2}})\right) \label{eq:App P1 k-1}\\
  \ubar_{\Kone}^{n+1} 
&=\textstyle\ubar_{\Kone}^n-\frac{1}{2}\left(u_{\Kone}^n(x_{\Kcut})-u_{k-1}^n(x_{\Kcut})\right)-\lambda h \nabla u_{k-1}^n \label{eq:App P1 k}\\
\ubar_{\Ktwo}^{n+1} &=\ubar_{\Ktwo}^n-\frac{\lambda}{1-\alpha}\left(u_{\Ktwo}^n(x_{k+\frac{1}{2}})-\frac{\alpha}{2\lambda}u_{\Kone}^n(x_{\Kcut})-\Big(1-\frac{\alpha}{2\lambda}\Big)u_{k-1}^n(x_{\Kcut})\right). \label{eq:App P1 k+1}
\end{align}
\end{subequations}

We decompose the sum of the TV in the means at time $t^{n+1}$ in terms $T_1$ to $T_5$:
\begin{align*}\textstyle
\sum_j \lvert \ubar_{\jj}^{n+1}-\ubar_{\jj-1}^{n+1} \rvert =& \underbrace{\textstyle\sum_{\jj\le \kk-1} \lvert \ubar_{\jj}^{n+1}-\ubar_{\jj-1}^{n+1} \rvert }_{T_1}+\underbrace{ \lvert \ubar_{\Kone}^{n+1}-\ubar_{\kk-1}^{n+1} \rvert }_{T_2}+\underbrace{ \lvert \ubar_{\Ktwo}^{n+1}-\ubar_{\Kone}^{n+1} \rvert}_{T_3}\\
&+\underbrace{ \lvert \ubar_{\kk+1}^{n+1}-\ubar_{\Ktwo}^{n+1} \rvert }_{T_4}+\underbrace{\textstyle\sum_{\jj\geq \kk+2} \lvert \ubar_{\jj}^{n+1}-\ubar_{\jj-1}^{n+1} \rvert}_{T_5}.
\end{align*}
In the following, we will estimate each of the terms $T_1$ to $T_5$
separately. We will use that the MC limiter guarantees
\begin{equation}\label{eq: prop MC limiter}
0\leq\frac{\nabla u_\jj}{\frac{2}{h_\jj}(\ubar_{\jj+1}^n-\ubar_\jj^n)}\leq 1, \quad
0\leq\frac{\nabla u_\jj}{\frac{2}{h_\jj}(\ubar_{\jj}^n-\ubar_{\jj-1}^n)}\leq 1,
\end{equation}
which implies
  \begin{align}
    \label{eq: P1_j}
\frac{u_j^n(x_{j+\frac{1}{2}})-u_{j-1}^n(x_{j-\frac{1}{2}})}{\ubar_j^n-\ubar_{j-1}^n}
&=1+\frac{\frac{h_j}{2}\nabla u_j}{\ubar_j^n-\ubar_{j-1}^n}-\frac{\frac{h_{j-1}}{2}\nabla u_{j-1}}{\ubar_j^n-\ubar_{j-1}^n} \quad \in\left[0,2\right].
\end{align}
Using this property and $\lambda < \frac 1 2$ (to guarantee the non-negativity of the pre-factors), there holds for two equidistant cells away from cells $\Kone$ and $\Ktwo$
\begin{multline*}
  \left\lvert\ubar_j^{n+1}-\ubar_{j-1}^{n+1}\right\rvert \leq
                                                               \left(1-\lambda\frac{u_j^n(x_{j+\frac{1}{2}})-u_{j-1}^n(x_{j-\frac{1}{2}})}{\ubar_j^n-\ubar_{j-1}^n}\right)\left\lvert\ubar_j^n-\ubar_{j-1}^n\right\rvert\\
+\lambda\frac{u_{j-1}^n(x_{j-\frac{1}{2}})-u_{j-2}^n(x_{j-\frac{3}{2}})}{\ubar_{j-1}^n-\ubar_{j-2}^n}\left\lvert\ubar_{j-1}^n-\ubar_{j-2}^n\right\rvert.
\end{multline*}
This yields bounds for $T_1$ and $T_5$:
\begin{align*}
T_1 &\leq \sum_{\jj\leq \kk-2}\left\lvert\ubar_j^n-\ubar_{j-1}^n\right\rvert +\left(1-\lambda\frac{u_{k-1}^n(x_{k-\frac{1}{2}})-u_{k-2}^n(x_{k-\frac{3}{2}})}{\ubar_{k-1}^n-\ubar_{k-2}^n}\right)\left\lvert\ubar_{k-1}^n-\ubar_{k-2}^n\right\rvert ,\\
T_5 &\leq \sum_{\jj\geq k+2}\left\lvert\ubar_j^n-\ubar_{j-1}^n\right\rvert +\lambda\frac{u_{k+1}^n(x_{k+\frac{3}{2}})-u_{\Ktwo}(x_{k+\frac{1}{2}}))}{\ubar_{k+1}^n-\ubar_{\Ktwo}^n}\left\lvert\ubar_{k+1}^n-\ubar_{\Ktwo}^n\right\rvert.
\end{align*}

For $T_2$ we can show using \eqref{eq:App P1 k-1} and \eqref{eq:App P1 k}
\begin{multline*}
T_2\leq \left(1-\frac{1}{2}\frac{u_{\Kone}^n(x_{\Kcut})-u_{\kk-1}^n(x_{\Kcut})}{\ubar_{\Kone}^n-\ubar_{\kk-1}^n}-\frac{\lambda h \nabla u_{k-1}}{\ubar_{\Kone}^n-\ubar_{\kk-1}^n}\right)\left\lvert\ubar_{\Kone}^n-\ubar_{k-1}^n\right\rvert\\
+\lambda \frac{u_{k-1}^n(x_{k-\frac{1}{2}})-u_{k-2}^n(x_{k-\frac{3}{2}})}{\ubar_{k-1}^n-\ubar_{k-2}^n}\left\lvert\ubar_{k-1}^n-\ubar_{k-2}^n\right\rvert.
\end{multline*}
For this we need to verify that the two
pre-factors are non-negative and that it is therefore allowed to pull them out of the absolute value.
The pre-factor of the second term is non-negative
due to \eqref{eq: P1_j}.
For the first pre-factor we obtain using \eqref{eq: cond limiting cut cell neigh text}
\begin{align*}
  \begin{split}
\frac{u_{\Kone}^n(x_{\Kcut})-u_{\kk-1}^n(x_{\Kcut})}{\ubar_{\Kone}^n-\ubar_{\kk-1}^n}
&=
1+\frac{\frac{\alpha h}{2}\nabla u_{\Kone}}{\ubar_{\Kone}^n-\ubar_{\kk-1}^n}-\frac{\left(\frac{h}{2}+\alpha h\right)\nabla u_{\kk-1}}{\ubar_{\Kone}^n-\ubar_{\kk-1}^n} \quad \in [0,2].
\end{split}
\end{align*}
This implies with \eqref{eq: prop MC limiter}, $\lambda < \frac{1}{4}$, and \eqref{eq: cond limiting cut cell neigh text}
\begin{multline*}
    1-\frac{1}{2}\frac{u_{\Kone}^n(x_{\Kcut})-u_{\kk-1}^n(x_{\Kcut})}{\ubar_{\Kone}^n-\ubar_{\kk-1}^n}-\frac{\lambda h \nabla u_{k-1}}{\ubar_{\Kone}^n-\ubar_{\kk-1}^n}
  \\
  =
    \frac{1}{2}-\frac{1}{2}\frac{\frac{\alpha h }{2}\nabla u_{\Kone}}{\ubar_{\Kone}^n-\ubar_{\kk-1}^n}
    +\frac{1}{2}\frac{\left(\left(\frac{h}{2}+\alpha h\right)-2\lambda h\right) \nabla u_{k-1}}{\ubar_{\Kone}^n-\ubar_{\kk-1}^n}
  \quad \in [0,1].
\end{multline*}

For $T_3$, we get from \eqref{eq:App P1 k} and \eqref{eq:App P1 k+1}
\begin{align*}
T_3\leq& \left(1-\frac{\lambda}{1-\alpha}\frac{u_{\Ktwo}^n(x_{\kk+\frac{1}{2}})-u_{\Kone}^n(x_{\Kcut})}{\ubar_{\Ktwo}^n-\ubar_{\Kone}^n}\right)\left\lvert\ubar_{\Ktwo}^n-\ubar_{\Kone}^n \right\rvert\\
&+\left(\left(\frac{1}{2}+\frac{\alpha-2\lambda}{2(1-\alpha)}\right)\frac{u_{\Kone}^n(x_{\Kcut})-u_{\kk-1}^n(x_{\Kcut})}{\ubar_{\Kone}^n-\ubar_{\kk-1}^n}+\frac{\lambda h \nabla u_{k-1}}{\ubar_{\Kone}^n-\ubar_{k-1}^n}\right)\left\lvert\ubar_{\Kone}^n-\ubar_{\kk-1}^n \right\rvert.
\end{align*}
Again we check the pre-factors.
The first pre-factor is obviously non-negative because of \eqref{eq: P1_j} and $\frac{\lambda}{1-\alpha}\leq \frac{1}{2}$.
Further, $\frac{1}{2}+\frac{\alpha-2\lambda}{2(1-\alpha)}\ge 0 $ for $\lambda \le \frac{1}{2}$. Therefore, using
\eqref{eq: prop MC limiter}, the second pre-factor is non-negative as well.
Finally, for $T_4$ we estimate
\begin{align*}
T_4 \leq& \left(1-\lambda\frac{u_{k+1}^n(x_{k+\frac{3}{2}})-u_{\Ktwo}(x_{k+\frac{1}{2}})}{\ubar_{k+1}^n-\ubar_{\Ktwo}^n}\right)\left\lvert \ubar_{k+1}^n-\ubar_{\Ktwo}^n\right\rvert\\
&+\frac{\lambda}{1-\alpha}\frac{u_{\Ktwo}^n(x_{k+\frac{1}{2}})-u_{\Kone}(x_{\Kcut})}{\ubar_{\Ktwo}^n-\ubar_{\Kone}^n}\left\lvert\ubar_{\Ktwo}^n-\ubar_{\Kone}^n\right\rvert\\
&+\left(\frac{\lambda}{1-\alpha}-\frac{\alpha}{2(1-\alpha)}\right)\frac{u_{\Kone}^n(x_{\Kcut})-u_{k-1}(x_{\Kcut})}{\ubar_{\Kone}^n-\ubar_{k-1}^n}\left\lvert \ubar_{\Kone}^n-\ubar_{k-1}^n\right\rvert.
\end{align*}
For the last term there holds
$
\frac{\lambda}{1-\alpha}-\frac{\alpha}{2(1-\alpha)}\geq 0 $ due to
$2\lambda\geq \alpha.
$
Together with other previously shown estimates, this implies that all three pre-factors are non-negative.
Finally, summing up the estimates for $T_1,\ldots,T_5$ implies the claim.
\end{proof}

\bibliographystyle{plain}
\bibliography{Literature}

\end{document}